\RequirePackage{fix-cm}
\RequirePackage{amsmath}
\RequirePackage{amssymb}

\documentclass[reqno,a4paper]{amsart}
\usepackage{amsmath}
\usepackage{amssymb}
\usepackage{mathrsfs}
\usepackage{bm}
\usepackage{amsmath,amssymb,amsfonts,bm}
\usepackage[colorlinks,linkcolor=blue,anchorcolor=red,citecolor=blue]{hyperref}
 \usepackage{geometry}
\newtheorem{Thm}{Theorem}[section]

\newtheorem{Lem}[Thm]{Lemma}
\newtheorem{Coro}[Thm]{Corollary}
\newtheorem{Rem}[Thm]{Remark}

\numberwithin{equation}{section}

\newcommand{\1}{\mathbf{1}}
\newcommand{\R}{\mathbb{R}}

\renewcommand{\P}{\mathbf{P}}

\newcommand{\E}{\mathcal{E}}
\newcommand{\D}{\mathcal{D}}
\newcommand{\ve}{\varepsilon}
\newcommand{\vt}{\vartheta}

\newcommand{\dis}{\displaystyle}
\newcommand{\pa}{\partial}
\newcommand{\na}{\nabla}
\newcommand{\de}{\delta}

\usepackage{cite}

\allowdisplaybreaks

\renewcommand{\>}{\rangle}
\newcommand{\T}{\mathbb{T}}
\newcommand{\I}{\mathbf{I}}
\renewcommand{\P}{\mathbf{P}}

\usepackage{textcomp}

\title{The Landau and Non-cutoff Boltzmann Equation in Union of Cubes}
\author[D.-Q. Deng]{Dingqun Deng    
}
\address{Beijing Institute of Mathematical Sciences and Applications, Tsinghua Univeristy, Beijing, People's Republic of China}
\email{dingqun.deng@gmail.com}

\begin{document}
		

	\begin{abstract}
		The existence and stability of collisional kinetic equation, especially non-cutoff Boltzmann equation, in bounded domain with physical boundary condition is longstanding open problem. This work proves the global stability of the Landau equation and non-cutoff Boltzmann equation in union of cubes with the specular reflection boundary condition when an initial datum is near Maxwellian.
		Moreover, the solution enjoys exponential large-time decay in bounded domain. 
		 Our method is based on that fact that normal derivatives in cubes is also derivatives along axis, which allows us to obtain high-order derivative estimates.

	\end{abstract}
\maketitle

	\tableofcontents
	
\section{Introduction}

\subsection{Equation and Domain}
	Boundary effects play a crucial role in the dynamics of collisional kinetic equation:
	\begin{align*}
		\pa_t F +v\cdot\na_x F =Q(F,F), \quad F(0,x,v)=F_0(x,v), 
	\end{align*}
where $F(t,x,v)$ denotes the particle distribution at time $t\ge 0$, position $x\in\Omega$ and velocity $v\in\R^3$ and $F_0$ denotes its initial datum. Throughout the paper, we consider the following two kinds of kinetic collision operator. 

\subsubsection{Landau collision operator} 
	For Landau collision operator, $Q$ is given by 
	\begin{align*}
		Q(G,F)&=\nabla_v\cdot\int_{\R^3}\phi(v-v')\big[G(v')\nabla_vF(v)-F(v)\nabla_vG(v')\big]\,dv'\\
		&=\sum^3_{i,j=1}\partial_{v_i}\int_{\R^3}\phi^{ij}(v-v')\big[G(v')\partial_{v_j}F(v)-F(v)\partial_{v_j}G(v')\big]\,dv'.
	\end{align*}
	The non-negative definite matrix-valued function $\phi=[\phi^{ij}(v)]_{1\leq i,j\leq 3}$ takes the form of 
	\begin{align*}
		\phi^{ij}(v) = \Big\{\delta_{ij}-\frac{v_iv_j}{|v|^2}\Big\}|v|^{\gamma+2},
	\end{align*}
where $\delta_{ij}$ is the Kronecker delta and $\gamma\ge -3$ is the interaction potential between particles. It is convenient to call it {\em hard potential} when $\gamma\ge -2$ and {\em soft potential} when $-3\le\gamma<-2$. The case $\gamma=-3$ corresponds to the physically realistic Coulomb interactions; cf. \cite{Guo2002a}. 
	
\subsubsection{Non-cutoff Boltzmann collision operator.} For Boltzmann collision operator without angular cutoff, $Q$ is defined by 	
	\begin{align*}
		Q(G,F) = \int_{\R^3}\int_{\mathbb{S}^{2}} B(v-v_*,\sigma)\big[G(v'_*)F(v')-G(v_*)F(v)\big]\,d\sigma dv_*.
	\end{align*}  
	In this expression $v,v_*$ and $v',v'_*$ are velocity pairs given in terms of the $\sigma$-representation by 
	\begin{align*}
		v'=\frac{v+v_*}{2}+\frac{|v-v_*|}{2}\sigma,\quad v'_*=\frac{v+v_*}{2}-\frac{|v-v_*|}{2}\sigma,\quad \sigma\in\mathbb{S}^2,
	\end{align*}
	that satisfy that 
	$v+v_*=v'+v'_*$ and 
	$|v|^2+|v_*|^2=|v'|^2+|v'_*|^2$.
	The Boltzmann collision kernel $B(v-v_*,\sigma)$ depends only on $|v-v_*|$ and the deviation angle $\theta$ through $\cos\theta=\frac{v-v_*}{|v-v_*|}\cdot\sigma$. Without loss of generality we can assume $B(v-v_*,\sigma)$ is supported on $0\le\theta\le\pi/2$, since one can reduce the situation with {\it symmetrization}: $\overline{B}(v-v_*,\sigma)={B}(v-v_*,\sigma)+{B}(v-v_*,-\sigma)$. Moreover, we assume 
	\begin{equation*}
		B(v-v_*,\sigma) = C_B|v-v_*|^\gamma b(\cos\theta),
	\end{equation*}
	for some $C_B>0$. $|v-v_*|^\gamma$ is called the kinetic part and $b(\cos\theta)$ is called the angular part. For non-cutoff Boltzmann case, we assume that there exist $C_b>0$ and $0<s<1$ such that  
	\begin{align*}
		\frac{1}{C_b\theta^{1+2s}}\le \sin\theta b(\cos\theta)\le \frac{C_b}{\theta^{1+2s}}, \quad\forall\,\theta\in (0,\frac{\pi}{2}].
	\end{align*} 
It is convenient call it {\em hard potential} when $\gamma+2s\ge 0$ and {\em soft potential} when $-3<\gamma+2s<0$. 

Throughout the paper, we assume $\gamma\ge -3$ for Landau case and $\gamma>\max\{-3,-\frac{3}{2}-2s\}$ for Boltzmann case. 
	
\subsubsection{Bounded Domain}

In this paper, we consider the bounded domain $\Omega\subset \R^3$ that is the union of finitely many cubes: 
\begin{align*}
	\Omega=\cup_{i=1}^N\Omega_i,
\end{align*} for some $N\ge 1$, where $\Omega_i = (a_{i,1},b_{i,1})\times(a_{i,2},b_{i,2})\times(a_{i,3},b_{i,3})$ with $a_{i,j}\in\R$. Note that $\Omega$ could be non-convex. Then
$\partial\Omega$ is divided into three kinds of boundary: $\pa\Omega=\cup^3_{i=1}\Gamma_i$, where $\Gamma_i$ is orthogonal to axis $x_i$ and is the union of finitely many connected sets. We also assume that $\Gamma_i$ is of {\em non-zero} spherical measure.  
Since the boundary of $\Gamma_i$'s is of {\em zero} spherical measure, we don't distinguish $\Gamma_i$ and the interior of $\Gamma_i$. 

The unit normal outward vector $n(x)$ exists on $\partial\Omega$ almost everywhere with respect to spherical measure. On the interior of $\Gamma_i$$(i=1,2,3)$, we have $
n(x) = e_i$ or $-e_i$, where $e_i$ is the unit vector with $i$th-component being $1$ and the other components being $0$.  We will denote vectors $\tau_1(x), \tau_2(x)$ on boundary $\partial\Omega$ such that $(n(x),\tau_1(x), \tau_2(x))$ form a unit orthonormal basis for $\R^3$. 
In this case, 
	\begin{align*}
		\int_{\Omega}\partial_{x_i}g\,dx = \int_{\partial\Omega}gn_i\,dS(x) = \int_{\Gamma_i}gn_i\,dS(x). 
	\end{align*}
The boundary of the phase space is 
\begin{align*}
	\gamma:=\{(x,v)\in\partial\Omega\times\R^3\}.
\end{align*}
With $n=n(x)$ being the outward normal direction at $x\in\partial\Omega$, we decompose $\gamma$ as 
\begin{align*}
	\gamma_- &= \{(x,v)\in\partial\Omega\times\R^3 : n(x)\cdot v<0\},\quad\text{(the incoming set),}\\
	\gamma_+ &= \{(x,v)\in\partial\Omega\times\R^3 : n(x)\cdot v>0\},\quad\text{(the outgoing set),}\\
	\gamma_0 &= \{(x,v)\in\partial\Omega\times\R^3 : n(x)\cdot v=0\},\quad\text{(the grazing set).}
\end{align*}
We consider, in this paper, the 
	{\em Specular reflection boundary condition:}
\begin{equation*}
	F(t,x,v) = F(t,x,R_xv), \text{ on } (x,v)\in \gamma_-, \text{ where }R_xv:= v-2n(x)(n(x)\cdot v).
\end{equation*}

There have been many contribution in the study of Boltzmann and Landau boundary value problems: \cite{Cercignani1992,Hamdache1992, Mischler2000,Yang2005, Liu2006,Guo2009, Esposito2013, Guo2016,Kim2017,Cao2019, Guo2020 }. 
For global stability in the perturbation framework, since the fundamental work by Guo \cite{Guo2009} on $L^2-L^\infty$ method, there are plenty of results developed for Boltzmann equation and Landau equation. For instance, Guo-Kim-Tonon-Trescases \cite{Guo2016} gave regularity of cutoff Boltzmann equation with several physical boundary conditions. Esposito-Guo-Kim-Marra \cite{Esposito2013} constructed a non-equilibrium stationary solution. Kim-Lee \cite{Kim2017} studied cutoff Boltzmann equation with specular boundary condition with external potential. Liu-Yang \cite{Liu2016} extended the result in \cite{Guo2009} to cutoff soft potential case. Cao-Kim-Lee \cite{Cao2019} proved the global existence for Vlasov-Poisson-Boltzmann with diffuse boundary condition. Guo-Hwang-Jang-Ouyang \cite{Guo2020} gave the global stability of Landau equation with specular reflection boundary. Duan-Liu-Sakamoto-Strain \cite{Duan2020} proved the global existence for Landau and non-cutoff Boltzmann equation in finite channel. 

Despite extensive developments in the study of Landau equation and Boltzmann equation, many basic boundary problems such as the global existence and uniqueness of non-cutoff Boltzmann equation in bounded domain have remain open. 

\subsubsection{Reformulation}
We write 
\begin{align*}
	F(t,x,v) = \mu(v) + \mu^{1/2}(v)f(t,x,v), 
\end{align*}
where $\mu$ is the global Maxwellian:
\begin{align*}
	\mu = (2\pi )^{-3/2}e^{-\frac{|v|^2}{2}}.
\end{align*}
Then the function $f$ satisfies
\begin{equation}
	\label{1}
	\partial_tf + v\cdot\na_xf = Lf +\Gamma(f,f)\quad f(0,x,v) = f_0(x,v). 
\end{equation}
where the linearized collision operator $L$ and nonlinear collision operator $\Gamma$ are given by 
\begin{align*}
	Lf = \mu^{-1/2}Q(\mu,\mu^{1/2}f)+\mu^{-1/2}Q(\mu^{1/2}f,\mu),
\end{align*}
and 
\begin{align*}
	\Gamma(f,f)=\mu^{-1/2}Q(\mu^{1/2}f,\mu^{1/2}f),
\end{align*}respectively. The specular boundary condition is given by 
\begin{equation}
	\label{specular}
	f(t,x,v) = f(t,x,R_xv), \text{ on } (x,v)\in \gamma_-, \text{ where }R_xv:= v-2n(x)(n(x)\cdot v).
\end{equation}
The kernel of $L$ is the span of $\{\mu^{1/2},v_i\mu^{1/2}$$(1\le i\le 3),|v|^2\mu^{1/2}\}$. Then we denote $\P$ to be the projection onto $\ker L$:
\begin{align}\label{abc}
	\P f = \big(a+b\cdot v+\frac{1}{2}c(|v|^2-3)\big)\mu^{1/2}(v),
\end{align}	
where 
\begin{align*}
	a = (f,\mu^{1/2})_{L^2_v},\quad b=(f,v\mu^{1/2})_{L^2_v},\quad c=\frac{1}{6}(f,(|v|^2-3)\mu^{1/2})_{L^2_v}.
\end{align*}



%

\subsection{Main Result}

Before presenting the main results, we specify some notations to be used through the paper.
Let $\<v\>=\sqrt{1+|v|^2}$ and 
$\partial^\alpha = \partial^{\alpha_1}_{x_1}\partial^{\alpha_2}_{x_2}\partial^{\alpha_3}_{x_3}$,
where $\alpha=(\alpha_1,\alpha_2,\alpha_3)$ is the multi-index. If each component of $\alpha_1$ is not greater than that of $\alpha$'s, we denote by $\alpha_1\le\alpha$. We will write $C>0$(large) and $\lambda>0$(small) to be generic constants, which may change from line to line. $\I$ is the identity mapping. $\1_{S}$ is the indicator function on a set $S$. Denote the $L^2_v$ and $L^2_{x,v}$, respectively, as 
\begin{align*}
	|f|^2_{L^2_v} = \int_{\R^3}|f|^2\,dv,\quad \|f\|_{L^2_{x,v}}^2 = \int_{\Omega}|f|^2_{L^2_v}\,dx.
\end{align*}
Denote $L^2_{B}$ to be the $L^2_v$ space inside a ball $B$. 
Define the weight function 
\begin{equation}
	\label{w2}
	w = w(v) = \exp\Big(\frac{q\<v\>^\vt}{4}\Big). 
\end{equation}
We assume the following condition on $q$ and $\vt$:
\begin{equation}\label{qvt}
	\left\{
	\begin{aligned}
		&q=0,\quad \text{ for hard potential in both Boltzmann and Landau case, }\\
		&\vt=1,\quad\text{ for soft potential in Boltzmann case,}\\
		&\vt\in[1,2] \text{ and retrict $q<1$ when $\vt=0$, for soft potential in Landau case.}
	\end{aligned}
	\right.
\end{equation}
For Landau equation, we denote 
\begin{align*}
	\sigma^{ij}(v)=\int_{\R^3}\phi^{ij}(v-v')\mu(v')\,dv',\quad \sigma^i(v)=\sum_{j=1}^3\int_{\R^3}\phi^{ij}(v-v')\frac{v'_j}{2}\mu(v')\,dv'.
\end{align*}
and the dissipation norm as 
\begin{align*}
	|f|^2_{L^2_{D,w}}= \int_{\R^3}w^2\Big(\sigma^{ij}\partial_{v_i}\partial_{v_j}+\sigma^{ij}\frac{v_i}{2}\frac{v_j}{2}|f|^2\Big)\,dv,\quad
	\|f\|^2_{L^2_xL^2_D} = \int_{\Omega}|f|^2_{L^2_D}\,dx.
\end{align*}
We also denote $|f|_{L^2_D} := |f|_{L^2_{D,1}}$. 
For non-cutoff Boltzmann equation, as in \cite{Gressman2011}, we denote
\begin{equation*}
	|f|^2_{L^2_D}:=|\<v\>^{\frac{\gamma+2s}{2}}f|^2_{L^2_v}+ \int_{\R^3}dv\,\<v\>^{\gamma+2s+1}\int_{\R^3}dv'\,\frac{(f'-f)^2}{d(v,v')^{3+2s}}\1_{d(v,v')\le 1},
\end{equation*}
and 
\begin{equation*}
	 |f|^2_{L^2_{D,w}}=|wf|^2_{L^2_D},\quad \|f\|^2_{L^2_xL^2_{D,w}} = \int_{\Omega}|f|_{L^2_{D,w}}^2\,dx.
\end{equation*}
The fractional differentiation effects are measured using the anisotropic metric on the {\it lifted} paraboloid
$d(v,v'):=\{|v-v'|^2+\frac{1}{4}(|v|^2-|v'|^2)^2\}^{1/2}$.

For the above norms, we denote 
\begin{align*}
	\|f\|^2_{H^2_xL^2_v} := \sum_{|\alpha|\le 2}\|\pa^\alpha f\|_{L^2_xL^2_v},\qquad
	\|f\|^2_{H^2_xL^2_D} := \sum_{|\alpha|\le 2}\|\pa^\alpha f\|_{L^2_xL^2_D}.
\end{align*}
To capture the energy estimate of Landau and non-cutoff Boltzmann equation, we introduce the ``instant energy functional" $\E(t)$ and the ``dissipation energy functional" $\D(t)$:
\begin{equation}\label{defe}
	\E(t) \approx \|f(t)\|^2_{H^2_xL^2_v}, 
\end{equation}	
\begin{equation}\label{defd}
	\D(t) := \| f(t)\|_{H^2_xL^2_D}^2.
\end{equation}
Also, we denote the weighted energy functional by 
\begin{equation}\label{defew}
	\E_w(t) \approx \|wf(t)\|^2_{H^2_xL^2_v},
\end{equation} and
\begin{equation}\label{defdw}
\D_w(t) := \| f(t)\|_{H^2_xL^2_{D,w}}^2.
\end{equation}
Note that for hard potential, we have $w=1$ and hence $\E(t)\approx\E_w(t)$ and $\D(t)=\D_w(t)$. 
Moreover, we can write the conservation laws on mass and energy as 
\begin{equation}
	\label{conservatrion}
	\int_{\T^3}\int_{\R^3}\begin{pmatrix}
		1    \\
		|v|^2  
	\end{pmatrix}\sqrt{\mu}f_0(x,v)\,dvdx=0.
\end{equation}
For large-time behavior, we define index 
\begin{equation}\label{p}
	p = \left\{
	\begin{aligned}
		&1,\qquad\qquad\qquad \text{ for hard potential in both Boltzmann and Landau case, }\\
		&\frac{1}{-\gamma-2s+1},\quad\text{ for soft potential in Boltzmann case,}\\
		&\frac{\vt}{-\gamma-2+\vt},\qquad \text{for soft potential in Landau case.}
	\end{aligned}
	\right.
\end{equation}

Next we present the main result of this article. 
\begin{Thm}\label{Main}
	Assume $\gamma>\max\{-3,-2s-\frac{3}{2}\}$ for Boltzmann case and $\gamma\ge -3$ for Landau case. Then there exists $\ve_0>0$ such that if $F_0(x,v)=\mu+\mu^{1/2}f_0(x,v)\ge 0$ satisfying \eqref{conservatrion} and
	\begin{align}
		\label{small}
		\sum_{|\alpha|\le 2}\|\pa^\alpha f_0\|_{L^2_xL^2_v}\le \ve_0, 
	\end{align}
	then there exists a unique global mild solution $f=f(t,x,v)$ to the problem \eqref{1} and \eqref{specular} satisfying that $F(t,x,v)=\mu+\mu^{1/2}f(t,x,v)\ge 0$ and for $T>0$, 
	\begin{equation*}
		\sup_{0\le t\le T}\E(t)+\int^T_0\D(t)\,dt\le \sum_{|\alpha|\le 2}\|\pa^\alpha f_0\|^2_{L^2_xL^2_v},
	\end{equation*}
	where $\E(t)$, $\D(t)$ are defined in \eqref{defe}, \eqref{defd} respectively. 
			For large-time behavior, we assume additionally 
			\begin{align}
				\label{small2}
				\sum_{|\alpha|\le 2}\|w\pa^\alpha f_0\|_{L^2_xL^2_v}\le \ve_0,
			\end{align}
		where $w$ is defined by \eqref{w2}. 
		Let $p\in(0,1]$ be given in \eqref{p}, then there exists $\delta>0$ such that the solution enjoys time decay estimate
			\begin{align}\label{timedecay}
				\|{wf}(t)\|_{H^2_xL^2_v} \lesssim e^{-\delta t^p}\|{wf_0 }\|_{H^2_xL^2_v}.
			\end{align}
\end{Thm}
We will make a few comments on Theorem \ref{Main}. Our main target throughout the paper is to study the global well-posedness for Landau equation and non-cutoff equation in union of cubes with physical boundary condition; namely, the specular reflection boundary condition. In union of finitely many cubes, one can define the normal derivatives as well as derivatives along axis on boundary by using the equation. They satisfy the specular reflection condition with sign $\pm$; see Lemma \ref{Lem24}. Then the boundary effect arising from $v\cdot\na_xf$ vanishes with this nice property.

As illustrated in \cite{Guo2016}, there should be singularity at the boundary if the domain is a ball and the singularity may propagate for in-flow injection, diffuse reflection and bounce-back reflection boundary conditions \cite{Kim2011}. 
Theorem \ref{Main} implies that Landau and Boltzmann equation have different behavior in cubes and in balls. Also, specular reflection boundary problem is different from other kinds of boundary conditions. In our case, we are able to construct $H^2_xL^2_v$ solutions.

	The reason for choosing boundary $\Omega$ being union of cubes is the following. Firstly, normal derivatives in cubes is also derivatives along axis. It follows that we can estimate the high-order derivatives on boundary as normal derivative. 
In order to deal with normal derivatives on boundary, we will apply 
\begin{align*}
	v\cdot\nabla_xf = v\cdot n(x)\partial_{n}f + v\cdot \tau_1(x)\partial_{\tau_1}f +v\cdot \tau_2(x)\partial_{\tau_2}f. 
\end{align*}
Together with equation \eqref{1}, we are able to define the boundary value for $\pa_nf$:
\begin{align*}
	\partial_nf = \frac{1}{v\cdot n}\Big(-v\cdot \tau_1(x)\partial_{\tau_1}f -v\cdot \tau_2(x)\partial_{\tau_2}f -\partial_tf +Lf +\Gamma(f,f)\Big).
\end{align*} 
One can deduce that $\pa_nf(x,R_xv)=-\pa_nf(x,R_xv)$ on $v\cdot n(x)\neq 0$. 
This is also called the compatible condition. 
When deriving the energy estimates with derivative $\pa_{x_i}$, it's necessary to note that $\pa_{x_i}$ is also the normal derivative on $\Gamma_i$. That is, $\pa_{x_i}f=\pa_nf$ or $\pa_{x_i}f=-\pa_nf$ on boundary $\Gamma_i$ and we can derive the specular reflection boundary condition for high-order derivative, which is frequently used in this paper; see Lemma \ref{Lem24}. 	

Next we give a short illustration for the vanishing boundary term. 
Taking inner product of $v\cdot\na_x\pa^\alpha f$ with test function $\Phi$ over $\Omega\times\R^3$, the boundary term occurs:
\begin{multline}\label{15}
	\int_{\partial\Omega}\int_{\R^3}v\cdot n(x)\pa^\alpha f(x,v)\Phi(x,v)\,dvdS(x)\\
	= \int_{\partial\Omega}\int_{\R^3}R_xv\cdot n(x)\pa^\alpha f(x,R_xv)\Phi(x,R_xv)\,dvdS(x).
\end{multline}
Although we have \eqref{specular}, $\pa^\alpha f(x,R_xv)$ has different properties when taking normal derivative and tangent derivatives; see \eqref{115c} and \eqref{115}. However, it's hard to evaluate $\partial^\alpha f(x,R_xv)$ on boundary for general bounded domain., 
Therefore, we choose $\Omega$ to be union of finitely many cubes to ensure that $\partial_{x_i}\in\{\pm\partial_n,\pm\pa_{\tau_1},\pm\pa_{\tau_2}\}$. Then one can apply \eqref{115c} and \eqref{115} to make sure \eqref{15} vanish with nicely chosen $\Phi$. 
%

The rest of the paper is organized as follows. In Section \ref{Sec2}, we give some basic estimates for linearized collision operator and nonlinear collision operator. In Section \ref{Sec3}, the macroscopic estimates for Landau and Boltzmann equation was derived. In Section \ref{Sec4}, we are able to prove the global existence with the {\em a priori} estimates and the local existence. In Section \ref{Sec5}, we give the proof of local existence for completeness. The Appendix \ref{Append} is devoted to Carleman representation for Boltzmann equation.

\section{Preliminary}\label{Sec2}
In this section, we provide several Lemmas on collision operator $L$ and $\Gamma(f,g)$. The first Lemma is concerned with weighted coercive estimate on $L$. The second Lemma is devoted to the trilinear estimate on $\Gamma(f,g)$ with velocity weight. 

\begin{Lem}
\label{L2}Assume $\gamma\ge -3$ for Landau case and $\gamma>\max\{-3,-2s-3/2\}$ for Boltzmann case. Then there exists decomposition for linearized collision operator
	\begin{equation*}
		L = -A+K,
	\end{equation*}such that 
\begin{align}\label{estiAK}
	(Af,f)_{L^2_v} \ge c_0|f|_{L^2_D}^2,
\end{align}	
\begin{align}\label{estiAK1}
(w^2Af,f)_{L^2_v} \ge c_0|f|_{L^2_{D,w}}^2-C|f|_{L^2_v},
\end{align}	
and $K$ is a bounded operator on $L^2_v$. Moreover, 
\begin{equation}
	\label{L}
	(-Lf,f)_{L^2_v}\ge c_0|\{\I-\P\}f|_{L^2_D}^2,
\end{equation}
and 
\begin{equation}\label{LL}
	(-w^2Lf,f)_{L^2_v}\ge c_0|f|_{L^2_{D,w}}^2- C|f|^2_{L^2_{B_C}},
\end{equation}
for some generic constant $c_0,C>0$.
\end{Lem}
\begin{proof} 
	The proof of \eqref{L} can be found in \cite[Lemma 5]{Guo2002a} for Landau case and \cite[(2.13)]{Gressman2011} for Boltzmann case. 
	The proof of \eqref{LL} can be found in \cite[Lemma 9]{Strain2007} for Landau case and \cite[Lemma 2.7]{Duan2013a} for Boltzmann case. 
	Note from \eqref{qvt} that when $\gamma\ge -2$ in Landau case and $\gamma+2s\ge 0$ in Boltzmann case, we have $q=0$ and hence, $w=1$. That is, it's not necessary to include any velocity weight in the {\em hard} potential cases. 
	\eqref{estiAK1} follows from the \eqref{LL} and the boundedness of $K$. Thus, we only prove \eqref{estiAK} and the boundedness of $K$ in the following. 
	
	For the proof of \eqref{estiAK}, we proceed in two cases. 
	
\medskip\noindent{\bf Case I: Landau equation.}
For Landau equation, we will apply the decomposition $L=A+K$ in \cite[Section 4.2]{Yang2016}. Let $\varepsilon>0$ small and choose a smooth cutoff function $\chi(|v|)\in[0,1]$ such that 
$
	\chi(|v|)=1\text{ if } |v|<\varepsilon;\  \chi(|v|)=0 \text{ if } |v|>2\varepsilon.
$
Then we can split $L=-A+K$ with   
\begin{equation}\label{AK1}\begin{aligned}
		-Af &= \partial_{v_i}(\sigma^{ij}\partial_{v_j}f) - \sigma^{ij}\frac{v_i}{2}\frac{v_j}{2}f
		+\partial_{v_i}\sigma^i\1_{|v|> R}f+A_1f\\
		&\qquad+(K_1-\1_{|v|\le R}K_1\1_{|v|\le R})f,\\
		Kf &= \partial_{v_i}\sigma^i\1_{|v|\le R}f + \1_{|v|\le R}K_1\1_{|v|\le R}f,
	\end{aligned}
\end{equation}
where $R>0$ is to be chosen large, $\varepsilon>0$ is to be chosen small, and $A_1$ and $K_1$ are respectively given by 
\begin{align*}
	A_1f &= -\mu^{-1/2}\partial_{v_i}\Big\{\mu\Big[\Big(\phi^{ij}\chi\Big)*\Big(\mu\partial_{v_j}\big[\mu^{-1/2}f\big]\Big)\Big]\Big\},\\
	K_1f &=  -\mu^{-1/2}\partial_{v_i}\Big\{\mu\Big[\Big(\phi^{ij}\big(1-\chi\big)\Big)*\Big(\mu\partial_{v_j}\big[\mu^{-1/2}f\big]\Big)\Big]\Big\},
\end{align*}
with the convolution taken with respect to the velocity variable $v$. 
Here and below repeated indices are implicitly summed over. 
From \cite[Lemma 3]{Guo2002a}, we know that  
\begin{align}\label{65a}
	|\partial_\beta\sigma^{ij}(v)|+|\partial_\beta\sigma^i(v)|\le C_\beta(1+|v|)^{\gamma+2-|\beta|}.
\end{align}
Then \cite[(4.33)]{Yang2016} shows that 
\begin{equation*}
	(A f,f)_{L^2_{v}}
	\ge c_0|f|^2_{L^2_{D}}.
\end{equation*}
Also, \cite[(4.32)]{Yang2016} and \eqref{65a} implies that $K$ is a bounded operator on $L^2_v$. 

\medskip\noindent{\bf Case II: Boltzmann equation.} 
We will use Pao's splitting as $\tilde{{\nu}}(v)={\nu}(v)+{\nu}_K(v)$; cf. \cite[p.568 eq. (65), (66)]{Pao1974} and \cite{Gressman2011}. Then $\nu(v)\approx\<v\>^{\gamma+2s}$ and $|{\nu}_K(v)|\lesssim \<v\>^\gamma$. We split $L=-A+K$ with 
\begin{equation}\label{AK}
	\begin{aligned}
		-Af &= \Gamma(\mu^{1/2},f)+{\nu}_K(v)f-\nu_K(v)\1_{|v|> R}f,\\
		Kf &= \Gamma(f,\mu^{1/2})-\nu_K(v)\1_{|v|\le R}f,
	\end{aligned}
\end{equation}
Then \cite[Lemma 2.4 and Lemma 2.5]{Gressman2011} show that 
\begin{align*}
	((A-\nu_K(v)\1_{|v|> R})f,f)_{L^2_v}&\approx |f|^2_{L^2_D}.
\end{align*}
Here, by using $|\<v\>^{\frac{\gamma+2s}{2}}(\cdot)|_{L^2_v}\lesssim |\cdot|_{L^2_D}$, we have 
\begin{align*}
	|(\nu_K(v)\1_{|v|> R}f,f)_{L^2_v}|\le C\<R\>^{-2s}|\<v\>^{\frac{\gamma+2s}{2}}f|_{L^2_v}^2\le C\<R\>^{-2s}|f|_{L^2_D}^2. 
\end{align*}
Then choosing $R>0$ large enough, we have 
\begin{equation*}
	(Af,f)_{L^2_v}\gtrsim |f|^2_{L^2_D}.
\end{equation*}
Using \cite[Lemma 2.1]{Global2019}, we know that $|\Gamma(f,\mu^{1/2})|_{L^2_v}\lesssim |\<v\>^{-C}f|_{L^2_v}$ for any $C>0$ and hence, 
\begin{align*}
	|Kf|_{L^2_v}\lesssim |\<v\>^{-C}f|_{L^2_v}. 
\end{align*}
This implies that $K$ is bounded on $L^2_v$ and completes the proof of Lemma \ref{L2}. 
\end{proof}

\begin{Lem}\label{gam}
Assume $\gamma\ge -3$ for Landau case and $\gamma>\max\{-3,-2s-3/2\}$ for Boltzmann case. Then
\begin{align}
	\label{gamma}
	(w^2\Gamma(f,g),h)_{L^2_v}\lesssim \Big(|wf|_{L^2_v}|g|_{L^2_{D,w}}+|f|_{L^2_{D,w}}|wg|_{L^2_v}\Big)|h|_{L^2_{D,w}}. 
\end{align}
Moreover, for any $|\alpha|\le 2$, we have 
\begin{align}\label{gammax}
	|(w^2{\partial^\alpha\Gamma(f,g)}, h)_{L^2_{x,v}}|\lesssim \Big(\|{wf}\|_{H^2_xL^2_v}\|{g}\|_{H^2_xL^2_{D,w}}+\|{f}\|_{H^2_xL^2_{D,w}}\|{wg}\|_{H^2_xL^2_{v}}\Big)\| {wh}\|_{L^2_xL^2_{D}}
\end{align}
\end{Lem}
\begin{proof}
	The proof of \eqref{gamma} can be found in \cite[Theorem 3; Lemma 10]{Guo2002a, Strain2007} and \cite[ Lemma 2.3; Lemma 2.4; (6.6)]{Gressman2011, Duan2013a, Fan2017}. 
Note that when $\gamma\ge -2$ in Landau case and $\Gamma+2s\ge 0$ in Boltzmann case, we have $q=0$ and hence $w=1$. That is, it's not necessary to include any velocity weight in the {\em hard} potential cases. 
To prove \eqref{gammax}, we apply \eqref{gamma} to estimate 
	\begin{align}\label{1.9}\notag
		&\quad\,\int_{\Omega}|({w^2\partial^\alpha\Gamma(f,g)},\partial^\alpha h)_{L^2_{v}}|\,dx\\
		&\lesssim \int_{\Omega}\sum_{\alpha_1\le\alpha}\Big(|{w\partial^{\alpha_1} f}|_{L^2_v}|{w\partial^{\alpha-\alpha_1} g}|_{L^2_{D}}+|{w\partial^{\alpha_1} f}|_{L^2_D}|{w\partial^{\alpha-\alpha_1} g}|_{L^2_{v}}\Big)\,dx\ \|w\partial^\alpha h\|_{L^2_xL^2_D}.
	\end{align}
Here we firstly consider the parts $|{\partial^{\alpha_1} f}|_{L^2_v}|{\partial^{\alpha-\alpha_1} g}|_{L^2_{D}}$. 
\begin{align*}
	\int_{\Omega}\sum_{\alpha_1\le\alpha}|w{\partial^{\alpha_1} f}|_{L^2_v}|w{\partial^{\alpha-\alpha_1} g}|_{L^2_{D}}\,dx
		&\notag\lesssim \sum_{|\alpha_1|=0}\|w{\partial^{\alpha_1} f}\|^2_{L^\infty_xL^2_v}\|w{\partial^{\alpha-\alpha_1} g}\|_{L^2_xL^2_{D}}\\
		&\notag\qquad+\sum_{|\alpha_1|=1}\|w{\partial^{\alpha_1} f}\|_{L^3_xL^2_v}\|w{\partial^{\alpha-\alpha_1} g}\|_{L^6_xL^2_{D}}\\
		&\notag\qquad+\sum_{|\alpha_1|=2}\|w{\partial^{\alpha_1} f}\|_{L^2_xL^2_v}\|w{\partial^{\alpha-\alpha_1} g}\|_{L^\infty_xL^2_{D}}\\
		&\lesssim \|{wf}\|_{H^2_xL^2_v}\|{wg}\|_{H^2_xL^2_{D}},
	\end{align*}
where we used embedding
 $\|f\|_{L^3_x(\Omega)}\lesssim\|f\|_{H^1_x(\Omega)}$, $\|f\|_{L^6_x(\Omega)}\lesssim\|\na_xf\|_{L^2_x(\Omega)}$ and $\|f\|_{L^\infty_x(\Omega)}\lesssim\|f\|_{H^2_x(\Omega)}$ from \cite[Section V and (V.21)]{Adams2003}. 
  Similarly, 
\begin{align*}
\int_{\Omega}\sum_{\alpha_1\le\alpha}|{\partial^{\alpha_1} f}|_{L^2_D}|{\partial^{\alpha-\alpha_1} g}|_{L^2_{v}}\,dx
&\lesssim \|{f}\|_{H^2_xL^2_D}\|{g}\|_{H^2_xL^2_{v}}.
\end{align*}
Plugging the above estimates into \eqref{1.9}, we obtain \eqref{gammax}. 	
\end{proof}

	\section{Macroscopic Estimates}\label{Sec3}

In this section we will derive the {\em a priori} estimates for the macroscopic part of a solution to the equation:
	\begin{align}\label{10}
		\partial_t{f}+v\cdot\nabla_x{f} - L {f} = g,\quad f|_{t=0}= f_0, 
	\end{align}
	with boundary condition \eqref{specular}, where $g$ is a chosen to be $0$ or $\Gamma(f,f)$.

To find the macroscopic dissipation, we take the following velocity moments
\begin{equation*}
	\mu^{\frac{1}{2}}, v_j\mu^{\frac{1}{2}}, \frac{1}{6}(|v|^2-3)\mu^{\frac{1}{2}},
	(v_j{v_m}-1)\mu^{\frac{1}{2}}, \frac{1}{10}(|v|^2-5)v_j \mu^{\frac{1}{2}}
\end{equation*}
with {$1\leq j,m\leq 3$} for the equation \eqref{10}. One sees that  
the coefficient functions $[a,b,c]=[a,b,c](t,x)$ satisfy the fluid-type system 
\begin{equation}\label{11}
	\left\{\begin{array}{l}
		\dis \pa_t a +\nabla_x \cdot b=0,\\
		\dis \pa_t b +\na_x (a+2c)+\na_x\cdot \Theta (\{\I-\P\} f)=0,\\[2mm]
		\dis \pa_t c +\frac{1}{3}\na_x\cdot b +\frac{1}{6}\na_x\cdot
		\Lambda (\{\I-\P\} f)=0,\\[2mm]
		\dis \pa_t[\Theta_{{ jm}}(\{\I-\P\} f)+2c\de_{{ jm}}]+\pa_jb_m+\pa_m
		b_j=\Theta_{jm}({r}+{h}),\\[2mm]
		\dis \pa_t \Lambda_j(\{\I-\P\} f)+\pa_j c = \Lambda_j({r}+{h}),
	\end{array}\right.
\end{equation}
where the
high-order moment functions $\Theta=(\Theta_{jm})_{3\times 3}$ and
$\Lambda=(\Lambda_j)_{1\leq j\leq 3}$ are respectively defined by
\begin{equation}
	\Theta_{jm}(f) = \left ((v_jv_m-1)\mu^{\frac{1}{2}}, f\right)_{L^2_v},\ \ \
	\Lambda_j(f)=\frac{1}{10}\left ((|v|^2-5)v_j\mu^{\frac{1}{2}},
	f\right)_{L^2_v},\notag
\end{equation}
with the inner product taken with respect to velocity variable $v$ only, and the terms ${r}$ and ${h}$ on the right are given by
\begin{equation*}
	{r}= -{v}\cdot \na_{{x}} \{\I-\P\}f,\ \ {h}=L \{\I-\P\}f+g.
\end{equation*}

Next, we derive the specular reflection boundary condition for high-order derivatives of solution to \eqref{1} and the boundary values for $[a,b,c]$. 
\begin{Lem}\label{Lem24}
	Let $f$ be the solution to \eqref{1} satisfying \eqref{specular}. Then we have the following identities on boundary $\big\{(x,v) : v\cdot n(x)\neq 0$ and $x$ belongs to the interior of $\Gamma_i(i=1,2,3)\big\}$:
	\begin{equation}\label{115q}
		f(x,v)=f(x,R_xv),
	\end{equation}
	and
	\begin{equation}\label{115c}\begin{aligned}
			\partial_{\tau_j}f(x,R_xv) &= \partial_{\tau_j}f(x,v),\\
			\partial_{\tau_{j}\tau_k}f(x,R_xv) &= \partial_{\tau_j\tau_k}f(x,v),
		\end{aligned}
	\end{equation}for $j,k=1,2$, 
	where $(n(x),\tau_1(x),\tau_2(x))$ forms a unit normal basis in $\R^3$.
	Then for the normal derivatives, on $\big\{(x,v) : v\cdot n(x)\neq 0$ and $x$ belongs to the interior of $\Gamma_i(i=1,2,3)\big\}$, we have
	\begin{align}\label{115}
		\partial_{n}f(x,R_xv) &= -\partial_{n}f(x,v),\\
		\label{115b}
		\partial_{\tau_j}\partial_{n}f(x,R_xv) &= -\partial_{\tau_j}\partial_{n}f(x,v),
	\end{align}
	for $j=1,2$, 
	and
	\begin{equation}\label{115d}
		\partial^2_{n}f(x,R_xv) = \partial^2_{n}f(x,v).
	\end{equation}
\end{Lem}
\begin{proof}
	Recall that $R_xv=v-2n(x)(n(x)\cdot v)$ maps $\gamma_-$ to $\gamma_+$. Then by \eqref{specular}, we have that on $\Gamma_i$$(i=1,2,3)$, 
	\begin{equation*}
		f(x,v)=f(x,R_xv),\ \text{ on } n(x)\cdot v\neq 0.
	\end{equation*}
	Note that on $\Gamma_i$, $R_xv$ sends $v_i$ to $-v_i$ and keep the other component the same while $\partial_{\tau_j}(j=1,2)$ differentiate along direction $x_k$ with $k\neq i$. Then we have 
	\eqref{115c}. 
	Next we claim that 
	\begin{equation}\label{2.10}
		Lf(x,v)=Lf(x,R_xv)\text{ and }g(R_xv)=g(v),\ \text{ on } n(x)\cdot v\neq 0,
	\end{equation}for any $x$ belongs to the interior of $\Gamma_i$. 
	Indeed, by \eqref{115q}, it suffices to show that 
	\begin{equation*}
		\Gamma(g_1,g_2)(R_xv)=\Gamma(g_1(R_xv),g_2(R_xv)),
	\end{equation*} for any $g_1,g_2$.
	For the Boltzmann case, we apply the Carleman representation \eqref{Carleman} to find that 
	\begin{align*}
		\Gamma(g_1,g_2)(R_xv) &= \int_{\R^3_h}\int_{E_{0,h}}\tilde{b}(\alpha,h)\1_{|\alpha|\ge|h|}\frac{|\alpha+h|^{\gamma+1+2s}}{|h|^{3+2s}}\mu^{1/2}(R_xv+\alpha-h)\\&\qquad\qquad\qquad\times\big(g_1(R_xv+\alpha)g_2(R_xv-h)-g_1(R_xv+\alpha-h)g_2(R_xv)\big)\,d\alpha dh\\ &=\Gamma(g_1(R_xv),g_2(R_xv)),
	\end{align*}where we apply rotation $R_x^{-1}$ on $(\alpha,h)$. 
	For the Landau case, we will apply the formula from \cite[Lemma 1]{Guo2002a}:
	\begin{align}
		\label{Landau}\notag
		\Gamma(f,g) &= \partial_{v_i}\Big[\Big\{\phi^{ij}*[\mu^{1/2}f]\Big\}\partial_{v_j}g\Big]
		-\Big\{\phi^{ij}*\Big[\frac{v_i}{2}\mu^{1/2}f\Big]\Big\}\partial_jg\\
		&\quad-\partial_{v_i}\Big[\Big\{\phi^{ij}*[\mu^{1/2}\partial_{v_j}f]\Big\}g\Big]
		-\Big\{\phi^{ij}*\Big[\frac{v_i}{2}\mu^{1/2}\partial_jf\Big]\Big\}g.
	\end{align}
	Noticing $\partial_{v_i}g(R_xv)=-\partial_{v_i}(g(R_xv))$ on $\Gamma_i$, $i=1,2,3$ and $\partial_{v_i}g(R_xv)=\partial_{v_i}(g(R_xv))$ on $\Gamma_j$, $j\neq i$, one can deduce that on $\Gamma_i$, 
	\begin{align*}
		&\quad\,\sum_{j,k=1}^3\partial_{v_j}\Big[\Big\{\phi^{jk}*[\mu^{1/2}f]\Big\}\partial_{v_k}g\Big](R_xv)\\
		&= \sum_{k=1}^3\partial_{v_i}\Big[-\Big\{\phi^{ik}*[\mu^{1/2}f]\Big\}(R_xv)\partial_{v_k}g(R_xv)\Big]+ \sum_{j\neq i}\sum_k\partial_{v_j}\Big[\Big\{\phi^{jk}*[\mu^{1/2}f]\Big\}(R_xv)\partial_{v_k}g(R_xv)\Big]\\
		&= \partial_{v_i}\Big[\Big\{\phi^{ii}*[\mu^{1/2}f]\Big\}(R_xv)\partial_{v_i}(g(R_xv))\Big]+
		\sum_{k\neq i}\partial_{v_i}\Big[\Big\{\phi^{ik}*[\mu^{1/2}f(R_xv)]\Big\}\partial_{v_k}(g(R_xv))\Big]\\
		&\quad+
		\sum_{j\neq i}\partial_{v_j}\Big[\Big\{\phi^{ji}*[\mu^{1/2}f(R_xv)]\Big\}\partial_{v_i}(g(R_xv))\Big]+
		\sum_{j\neq i, k\neq i}\partial_{v_j}\Big[\Big\{\phi^{jk}*[\mu^{1/2}f]\Big\}(R_xv)\partial_{v_k}g(R_xv)\Big]\\
		&= \sum_{j,k=1}^3\partial_{v_j}\Big[\Big\{\phi^{jk}*[\mu^{1/2}f(R_xv)]\Big\}\partial_{v_k}(g(R_xv))\Big],
	\end{align*}where we used $\phi^{ik}(R_xv) = -\phi^{ik}(v)$, $\phi^{ji}(R_xv)=-\phi^{ji}(v)$ when $k\neq i$, $j\neq i$. Similar calculation can be applied to second to fourth term in \eqref{Landau}. Thus		
	$\Gamma(g_1,g_2)(R_xv)=\Gamma(g_1(R_xv),g_2(R_xv))$ on $\Gamma_i$$(i=1,2,3)$. This completes the claim.
	
	With the above claim, 
	using identity 
	\begin{equation}\label{111a}
		v\cdot\nabla_xf = v\cdot n(x)\partial_{n}f + v\cdot \tau_1(x)\partial_{\tau_1}f +v\cdot \tau_2(x)\partial_{\tau_2}f, 
	\end{equation} 
	we can apply equation \eqref{10} to define normal derivative $\partial_{n}f$ on interior of $\Gamma_i$:
	\begin{align*}
		v\cdot n\partial_nf = -v\cdot \tau_1(x)\partial_{\tau_1}f -v\cdot \tau_2(x)\partial_{\tau_2}f -\partial_tf +Lf +\Gamma(f,f).
	\end{align*}
	Then by \eqref{115c} and \eqref{2.10} we have 
	\begin{align*}
		R_xv\cdot n(x)\partial_{n}f(x,R_xv) = v\cdot n(x)\partial_{n}f(x,v), \quad \text{on }v\cdot n(x)\neq 0, \,\forall\,x\in\Gamma_i.
	\end{align*}This gives \eqref{115}. 
	When $g=\Gamma(f,f)$, we have $\partial_ng=\Gamma(\partial_nf,f)+\Gamma(f,\pa_nf)$. Again using Carleman representation \eqref{Carleman} for Boltzmann equation and \eqref{Landau} for Landau equation, we have from \eqref{115} that 
	\begin{equation}\label{115a}
		\partial_nLf(x,R_xv) = -\partial_nLf(x,v),\quad \partial_ng(x,R_xv) = -\partial_ng(x,v),
	\end{equation}on $v\cdot n(x)\neq 0$, for $x\in\Gamma_i$ and 
	\begin{equation*}\begin{aligned}
			R_xv\cdot \tau_1(x)\partial_{\tau_1}\partial_{n}f(x,R_xv) = -v\cdot \tau_1(x)\partial_{\tau_1}\partial_{n}f(x,v),\\ \quad R_xv\cdot \tau_2(x)\partial_{\tau_2}\partial_{n}f(x,R_xv) = -v\cdot \tau_2(x)\partial_{\tau_2}\partial_{n}f(x,v),
		\end{aligned}
	\end{equation*}on $v\cdot n(x)\neq 0$, for $x\in\Gamma_i$. These identities give \eqref{115b}. 
	Again we apply the equation \eqref{10} and \eqref{111a} to define second normal derivative $\partial^2_{n}f$ on interior of $\Gamma_i$:
	\begin{align*}
		v\cdot n\partial_n\partial_nf = -v\cdot \tau_1(x)\partial_{\tau_1}\partial_nf -v\cdot \tau_2(x)\partial_{\tau_2}\partial_nf -\partial_t\partial_nf +L\partial_nf +\pa_n g.
	\end{align*} 
	Then applying \eqref{115}, \eqref{115a} and \eqref{115b}, we have 
	\begin{align*}
		R_xv\cdot n(x)\partial^2_{n}f(x,R_xv) = -v\cdot n(x)\partial^2_{n}f(x,v), .
	\end{align*}$\text{on }v\cdot n(x)\neq 0$, for $x\in\Gamma_i$. This completes the proof of Lemma \ref{Lem24}
\end{proof}

\begin{Rem}
	
	(1)	Note that \eqref{115q}, \eqref{115c}, \eqref{115}, \eqref{115b}, \eqref{115d} are only valid on the interior of $\Gamma_i$. However, the intersection of $\Gamma_i$'s (the boundary of $\Gamma_i$'s) is of zero spherical measure and hence, the integration on those intersection doesn't influence the whole boundary integration $\int_{\partial\Omega}= \sum_i\int_{\Gamma_i}$. 
	
	(2) Note that we are using the equation to define boundary value for $f$. One can also assume these boundary conditions initially and regard them as compatible conditions, since they are satisfied if the solution exists. 
\end{Rem}
As a corollary, by definition \eqref{abc} for $[a,b,c]$, we have the following boundary value.
\begin{Coro}\label{Lem25}
	For $i=1,2,3$ and any $x$ belongs to interior of $\Gamma_i$, we have 
	\begin{align}\label{Lem25a}
		\partial_{x_i}c(x) =\partial_{x_i}a(x)=\partial_{x_i}b_j(x)=	b_i(x) = 0,
	\end{align}for $j\neq i$. 
	As a consequence, 
	\begin{equation}\label{2.14}
		\begin{aligned}
			\sum_{i,j=1}^3\|\partial_{x_ix_j}a\|^2_{L^2_x} = \|\Delta_xa\|^2_{L^2_x},\\ \sum_{i,j=1}^3\|\partial_{x_ix_j}b\|^2_{L^2_x} = \|\Delta_xb\|^2_{L^2_x},\\ \sum_{i,j=1}^3\|\partial_{x_ix_j}c\|^2_{L^2_x} = \|\Delta_xc\|^2_{L^2_x}. 
		\end{aligned}
	\end{equation}
\end{Coro}
\begin{proof}
	Notice that $\partial_nf = \pm\partial_{x_i}f$ on $\Gamma_i$. Then
	by \eqref{115} and change of variable $v\mapsto R_xv$, we have on interior of $\Gamma_i$ that 
	\begin{equation*}
		\partial_{x_i}c = \int_{\R^3}\partial_{x_i}f(x,R_xv)|R_xv|^2\mu^{1/2}(R_xv)\,dv
		= -\int_{\R^3}\partial_{x_i}f(x,v)|v|^2\mu^{1/2}(v)\,dv = 0. 
	\end{equation*}
	Similarly, on interior of $\Gamma_i$, we have 
	\begin{equation*}
		\partial_{x_i}a = \int_{\R^3}\partial_{x_i}f(x,R_xv)\mu^{1/2}(R_xv)\,dv
		= -\int_{\R^3}\partial_{x_i}f(x,v)\mu^{1/2}(v)\,dv = 0.  
	\end{equation*}
	For $j\neq i$, noticing $(R_xv)_j=v_j$ on $\Gamma_i$, we have 
	\begin{equation*}
		\partial_{x_i}b_j = \int_{\R^3}\partial_{x_i}f(x,R_xv)(R_xv)_j\mu^{1/2}(R_xv)\,dv
		= -\int_{\R^3}\partial_{x_i}f(x,v)v_j\mu^{1/2}(v)\,dv = 0.  
	\end{equation*}
	On interior of $\Gamma_i$, we have $(R_xv)_i=-v_i$ and hence by \eqref{115q}, 
	\begin{equation*}
		b_i(x) = \int_{\R^3}f(x,R_xv)(R_xv)_i\mu^{1/2}(R_xv)\,dv
		= -\int_{\R^3}f(x,v)v_i\mu^{1/2}(v)\,dv = 0. 
	\end{equation*}
	For $\eqref{2.14}$, notice that for $i\neq j$, $\partial_{x_ix_j}a = 0$ on $\Gamma_i$ and $\partial_{x_j}a=0$ on $\Gamma_j$. Then 
	\begin{align*}
		\int_{\Omega}|\partial_{x_ix_j}a|^2\,dx
		&= \int_{\Gamma_i}\partial_{x_ix_j}a\,\partial_{x_j}a\,dS(x) - \int_{\Gamma_j}\partial_{x_ix_i}a\,\partial_{x_j}a\,dx + \int_{\Omega}\partial_{x_ix_i}a\,\partial_{x_jx_j}a\,dx
		\\&= \int_{\Omega}\partial_{x_ix_i}a\partial_{x_jx_j}a\,dx, 
	\end{align*}
	where $dS$ is the spherical measure. Then we have $\sum_{i,j}\|\partial_{x_ix_j}a\|^2_{L^2_x} = \|\Delta_xa\|^2_{L^2_x}$. Similar argument can be applied to $c$ and one can deduce \eqref{2.14}$_1$ and \eqref{2.14}$_3$. 
	For $\eqref{2.14}_2$, notice that for $j\neq i$, we have 
	\begin{align*}
		\partial_{x_ix_j}b_k= 0 \text{  or  }\partial_{x_j}b_k=0, \text{  on }\Gamma_i,\\
		\partial_{x_ix_i}b_k= 0 \text{  or  }\partial_{x_j}b_k=0, \text{  on }\Gamma_j.
	\end{align*}
	Then we have 
	\begin{align*}
		\int_{\Omega}|\partial_{x_ix_j}b_k|^2\,dx
		&= \int_{\Gamma_i}\partial_{x_ix_j}b_k\,\partial_{x_j}b_k\,dS(x) - \int_{\Gamma_j}\partial_{x_ix_i}b_k\,\partial_{x_j}b_k\,dx + \int_{\Omega}\partial_{x_ix_i}b_k\,\partial_{x_jx_j}b_k\,dx
		\\&= \int_{\Omega}\partial_{x_ix_i}b_k\,\partial_{x_jx_j}b_k\,dx. 
	\end{align*}This implies \eqref{2.14}$_2$ and completes Corollary \ref{Lem25}. 
\end{proof}	
	
	We denote $\zeta(v)$ to be a smooth function satisfying $$
	\zeta(v)\lesssim e^{-\lambda|v|^2},$$ for some $\lambda >0$. The function $\zeta(v)$ may change from line to line. The following integration will be used frequently: if $p>-1$ is an even number, then 
	\begin{align*}
		\int_{\R}z^pe^{-\frac{|z|^2}{2}}dz = (p-1)!!\sqrt{2\pi}.
	\end{align*}
Next we write the main dissipation estimates for macroscopic parts. 
	\begin{Thm}\label{Thm31}
		Assume $\gamma>\max\{-3,-2s-\frac{3}{2}\}$ for Boltzmann case and $\gamma\ge -3$ for Landau case. 
		Then there exists a functional $\E_{int}(t)$ satisfying 
		\begin{align*}
			\E_{int}(t)\lesssim \sum_{|\alpha|\le 2}\|\partial^\alpha f\|_{L^2_xL^2_v},
		\end{align*} such that 
		\begin{align}\label{121}
			\partial_t\E_{int}(t) + \lambda\sum_{|\alpha|\le 2}\|\partial^\alpha [{a},{b},{c}]\|^2_{L^2_{x}}
			\lesssim \sum_{|\alpha|\le 2}\|({\partial^\alpha g},\zeta)_{L^2_v}\|^2_{L^2_{x}}+\sum_{|\alpha|\le 2}\|\{\I-\P\}{\partial^\alpha  f}\|^2_{L^2_{x}L^2_D},
		\end{align}for some $\lambda>0$. 
	\end{Thm}
	\begin{proof}
		
		Let $|\alpha|\le 2$ and $\partial^\alpha=\partial_{x_ix_i}$$(i=1,2,3)$ if $|\alpha|=2$. Notice that from \eqref{2.14}, we only need to deal with derivatives $\partial_{x_ix_i}$ when estimating the second derivatives in $\|\pa^\alpha[{a},{b},{c}]\|^2_{L^2_{x}}$. Applying $\partial^\alpha$ to \eqref{10}, we have 
		\begin{align}\label{11a}
			\partial_t{\partial^\alpha f}+v\cdot\nabla_x{\partial^\alpha f} &- L {\partial^\alpha f} = {\partial^\alpha  g}.
		\end{align}	
		Let ${\Phi}(t,x,v)\in C^1((0,+\infty)\times\Omega\times\R^3)$ be a test function. Taking the inner product of ${\Phi}(t,x,v)$ and \eqref{11a} with respect to $(x,v)$,
we obtain 
		\begin{align*}
			&\quad\,\partial_t({\partial^\alpha f},{\Phi})_{L^2_{x,v}}(t)- ({\partial^\alpha f},\partial_t{\Phi})_{L^2_{x,v}}-({\partial^\alpha f},v\cdot{\nabla_{x}\Phi})_{L^2_{x,v}} 
			\\
			&+\int_{\partial\Omega}(v\cdot n(x){\partial^\alpha f}(x),{\Phi}(x))_{L^2_v}\,dS(x) - (L {\partial^\alpha f},{\Phi})_{L^2_{x,v}} = ({\partial^\alpha g},{\Phi})_{L^2_{x,v}},
		\end{align*}
	where $dS(x)$ is the spherical measure. 
		Using the decomposition ${f}=\P{f}+\{\I-\P\}{f}$, we have 
		\begin{align}\label{100}
			\partial_t({\partial^\alpha f},{\Phi})_{L^2_{x,v}}(t)-({\partial^\alpha \P f},v\cdot{\nabla_{x}\Phi})_{L^2_{x,v}}  = \sum_{j=1}^4S_j,
		\end{align}
		where $S_j$ are defined by 
		\begin{align*}
			 S_1 &= ({\partial^\alpha f},\partial_t{\Phi})_{L^2_{x,v}},\\
			S_2 &= ({\partial^\alpha \{\I-\P\}f},v\cdot{\nabla_{x}\Phi})_{L^2_{x,v}} ,\\
			S_3&= (L {\partial^\alpha f},{\Phi})_{L^2_{x,v}}+({\partial^\alpha g},{\Phi})_{L^2_{x,v}},\\
			S_4 &= -\int_{\partial\Omega}(v\cdot n(x){\partial^\alpha f}(x),{\Phi}(x))_{L^2_v}\,dS(x).
		\end{align*}

		\medskip \noindent{\bf Estimate on ${c}(t,x)$:} We choose the following test function 
		\begin{align*}
			{\Phi} = {\Phi_c} = (|v|^2-5)\big(v\cdot{\nabla_{x}\phi_c}(t,x)\big)\mu^{1/2},
		\end{align*}
		where 				
		\begin{equation}\label{120}\left\{\begin{aligned}
				&-\Delta_x \phi_c = {\partial^\alpha c},\\
				&{\phi_c}(x)= 0 \ \text{ on }\ x\in \Gamma_i,\ \text{ if }\alpha_i = 1,\\
				&\frac{\partial\phi_c}{\partial n}(x)= 0\ \text{ on }\ x\in \Gamma_i,\ \text{ if }\alpha_i = 0\text{ or } 2.
			\end{aligned}\right.
		\end{equation}
	The existence and uniqueness of solution to \eqref{120} is guaranteed by \cite[Lamma 4.4.3.1]{Grisvard1985}. In particular, when $|\alpha|=0$ or $\alpha_i=2$ for some $i$, \eqref{120} is pure Neumann problem and we need
	$\int_{\Omega}c\,dx=0$ and  $\int_{\Omega}\partial_{x_ix_i}c\,dx=\int_{\Gamma_i}\partial_{x_i}c\,dS(x)=0$ respectively to ensure the existence of \eqref{120}, which follows from \eqref{conservatrion} and \eqref{Lem25a}. 
	Similar to the proof for \eqref{2.14}, by using boundary value of $\phi_c$, we have 
	\begin{align}\label{3.8}
		\sum_{i,j=1}^3\|\partial_{x_ix_j}{\phi_c}\|_{L^2_{x}}^2 = \|\Delta_x\phi_c\|_{L^2_x}^2 \lesssim \|\partial^\alpha c\|^2_{L^2_x}. 
	\end{align}
Here the second inequality follows from equation \eqref{120}. 
We will discuss the value of $\alpha$ in two cases. 

If $|\alpha|=0$, then \eqref{120} is a pure Neumann boundary problem and the solution is unique up to a constant. Thus, we can choose the constant carefully such that 
$\int_\Omega \phi_c\,dx=0$. Then by Poincar\'{e}'s inequality, we have 
\begin{align*}
	\|\phi_c\|_{L^2_x}\lesssim \|\na_x\phi_c\|_{L^2_x}. 
\end{align*}
By standard elliptic estimate of \eqref{120}, we have 
\begin{align}\label{3.18a}
	\|\na_x\phi_c\|_{L^2_x}^2=|(c,\phi_c)_{L^2_x}|\lesssim \|c\|_{L^2_x}\|\na_x\phi\|_{L^2_x}.
\end{align}
This implies that 
\begin{align}\label{3.9a}
	\|\na_x\phi_c\|_{L^2_x}\lesssim \|c\|_{L^2_x}. 
\end{align}
Similarly, since $\pa_t$ doesn't affect the boundary value for $\phi_c$, we have 
\begin{align}\label{3.10a}
	\|\pa_t\na_x\phi_c\|_{L^2_x}\lesssim \|\pa_tc\|_{L^2_x}\lesssim \sum_{|\alpha|=1}\Big(\|\partial^\alpha b\|_{L^2_x}+\|\pa^\alpha\{\I-\P\}f\|_{L^2_xL^2_D}\Big). 
\end{align}

If $|\alpha| = 1$, then $\alpha_i=1$ for some $1\le i\le 3$ and $\phi_c(x)=0$ on $\Gamma_i$.
Using the boundary value for $\phi_c$, i.e. $\partial_{x_j}\phi_c=0$ or $\phi_c=0$ on $\Gamma_j$ for any $j$, we have 
\begin{align}\label{3.9}\notag
	\|\na_x\phi_c\|^2_{L^2_x} &= \sum_{j=1}^3\int_{\Gamma_j}\partial_{x_j}\phi_c\,\phi_c\,dx - \int_{\Omega}\Delta_x\phi_c\,\phi_c\,dx\\
	&= \int_{\Omega}\partial^\alpha c\,\phi_c\,dx \le \|\partial^\alpha c\|_{L^2_x}\|\phi_c\|_{L^2_x}.
\end{align}
Since $\phi_c =0$ on $\Gamma_i$, by \cite[Theorem 6.7-5]{Ciarlet2013}, we have $\|\phi_c\|_{L^2_x}\lesssim \|\na_x\phi_c\|_{L^2_x}$. Then from \eqref{3.9}, we have 
\begin{align}\label{3.10}
	\|\na_x\phi_c\|_{L^2_x}\lesssim \|\partial^\alpha c\|_{L^2_x}\lesssim \sum_{|\alpha|=1}\|\partial^\alpha c\|_{L^2_x}.
\end{align}
Similarly, since derivative on time $t$ doesn't affect the boundary value, we have 
\begin{align}\label{3.11}
	\|\partial_t\na_x\phi_c\|_{L^2_x}\lesssim \sum_{|\alpha|=1}\|\partial_t\partial^\alpha c\|_{L^2_x}\lesssim \sum_{|\alpha|=2}\Big(\|\partial^\alpha b\|_{L^2_x}+\|\pa^\alpha\{\I-\P\}f\|_{L^2_xL^2_D}\Big), 
\end{align}
where the second inequality follows from \eqref{11}. 

If $|\alpha|=2$, as stated at the beginning of the proof, we only consider the case that $\alpha_i=2$ for some $1\le i\le 3$. In this case, \eqref{120} is a pure Neumann boundary problem. Then for this $i$, similar to \eqref{3.9}, by using boundary values $\partial_{x_i}c=0$ on $\Gamma_i$ from Corollary \ref{Lem25}, we have 
\begin{align*}
	\|\na_x\phi_c\|^2_{L^2_x}
	&= \int_{\Omega}\partial_{x_ix_i} c\,\phi_c\,dx = \int_{\Gamma_i}\partial_{x_i}c\,\phi_c\,dx - \int_{\Omega}\partial_{x_i}c\,\partial_{x_i}\phi_c\,dx
	\le \|\partial_{x_i}c\|_{L^2_x}\|\partial_{x_i}\phi_c\|_{L^2_x}.
\end{align*}
This implies that 
\begin{align}\label{3.12}
	\|\na_x\phi_c\|_{L^2_x}\le \|\partial_{x_i}c\|_{L^2_x}.
\end{align}
Similarly, noticing derivative on time $t$ doesn't affect the boundary value for $\phi_c$, we have 
\begin{align}\label{3.13}
	\|\partial_t\na_x\phi_c\|_{L^2_x}\le \|\partial_t\partial_{x_i}c\|_{L^2_x}\lesssim \sum_{|\alpha|=2}\Big(\|\partial^\alpha b\|_{L^2_x}+\|\pa^\alpha\{\I-\P\}f\|_{L^2_xL^2_D}\Big).
\end{align}

		Now we can compute \eqref{100}. For the second term on left hand side of \eqref{100}, we have 
		\begin{align*}
			&\quad\,-({\partial^\alpha \P f},v\cdot{\nabla_{x}\Phi_{c}})_{L^2_{x,v}} \\
			&= -\sum_{j,m=1}^3({\partial^\alpha a}+{\partial^\alpha b}\cdot v+\frac{1}{2}(|v|^2-3){\partial^\alpha c} ,v_jv_m(|v|^2-5)\mu{\partial_j\partial_m\phi_{c}})_{L^2_{x,v}} \\
			&= 5\sum_{j=1}^3({\partial^\alpha c} ,{-\partial^2_j\phi_{c}})_{L^2_{x,v}}  = 5\|{\partial^\alpha c}\|^2_{L^2_{x,v}} .
		\end{align*}
	Note that $\int_{\R^3}|v|^4v^2_j\mu\,dv=35$, $\int_{\R^3}|v|^2v^2_j\mu\,dv=5$ and  $\int_{\R^3}v^2_j\mu\,dv=1$.  
		For $S_1$, we see from \eqref{3.10a}, \eqref{3.11} and \eqref{3.13} that 
		\begin{align*}
			|S_1|&\le|({\partial^\alpha f},\partial_t{\Phi_c})_{L^2_{x,v}}| = |(\{\I-\P\}{\partial^\alpha f},\partial_t{\Phi_c})_{L^2_{x,v}}|\\
			&\lesssim \eta\|\partial_t{\nabla_x\phi_c}\|^2_{L^2_{x}}+C_\eta\|\{\I-\P\}{\partial^\alpha f}\|^2_{L^2_{x}L^2_D}\\
			&\lesssim \eta
			\sum_{1\le|\alpha|\le2}\|\partial^\alpha b\|^2_{L^2_{x}}+C_\eta\sum_{|\alpha|\le 2}\|\{\I-\P\}{\partial^\alpha f}\|^2_{L^2_{x}L^2_D}.
		\end{align*}
		Thanks to \eqref{3.8}, $S_2$ can be estimated as 
		\begin{align*}
			|S_2|&\lesssim \eta\sum_{i,j=1}^3\|\partial_{x_ix_j}\phi_c\|_{L^2_x}^2+C_\eta\|\partial^\alpha\{\I-\P\}f\|^2_{L^2_xL^2_D}\\&\lesssim \eta\|{\partial^\alpha c}\|^2_{L^2_{x}}+C_\eta\|\{\I-\P\}{\partial^\alpha f}\|^2_{L^2_{x}L^2_D}.
		\end{align*}
		For the term $S_3$, applying \eqref{3.9a}, \eqref{3.10} and \eqref{3.12}, we have 
		\begin{align*}
			|S_3|\le \eta\sum_{|\alpha|\le 1}\|{\partial^\alpha c}\|^2_{L^2_{x}}+C_\eta\sum_{|\alpha|\le 2}\|\{\I-\P\}{\partial^\alpha f}\|^2_{L^2_{x}L^2_D}+C_\eta\sum_{|\alpha|\le 2}\|({\partial^\alpha g},\zeta)_{L^2_v}\|^2_{L^2_{x}}.
		\end{align*}
		For $S_4$, we will use the boundary condition \eqref{specular}. 
	\begin{align*}
		S_4 &= -\int_{\partial\Omega}(v\cdot n(x){\partial^\alpha f}(x),{\Phi_c}(x))_{L^2_v}\,dS(x).
	\end{align*}
 We divide the integral on $\partial\Omega$ into three parts: $\Gamma_i$, $i=1,2,3$ and consider each $\Gamma_i$ separately. Fix $i=1,2,3$, then on $\Gamma_i$, we have $\partial_{n} = \partial_{x_i}$ or $-\pa_{x_i}$ and $(\tau_1,\tau_2)$ is the vector having components $x_j$, with $j\neq i$ and $1\le j\le 3$. Then 
 \begin{multline}
 	\int_{\Gamma_i}(v\cdot n(x){\partial^\alpha f}(x),{\Phi_c}(x))_{L^2_v}\,dS(x)\\
 	= \int_{\Gamma_i}\int_{\R^3}v\cdot n(x){\partial^\alpha f}(t,x,v)(|v|^2-5)\big(v\cdot{\nabla_{x}\phi_c}(t,x)\big)\mu^{1/2}\,dvdS(x).\label{20}
 \end{multline}
If $\alpha_i=0$ or $2$, then from boundary condition \eqref{120} we know that $\pa_{x_i}\phi_c=0$ on $\Gamma_i$. Applying change of variable $v\mapsto R_xv$, \eqref{20} becomes 
	\begin{align*}
		&\quad\,\int_{\Gamma_i}\int_{\R^3}v\cdot n(x){\partial^\alpha f}(t,x,v)(|v|^2-5)\sum_{j\neq i}\big(v_j\pa_{x_j}\phi_c(t,x)\big)\mu^{1/2}\,dvdS(x)\\
		&=\int_{\Gamma_i}\int_{\R^3}R_xv\cdot n(x){\partial^\alpha f}(t,x,R_xv)(|R_xv|^2-5)\sum_{j\neq i}\big((R_xv)_j\pa_{x_j}\phi_c(t,x)\big)\mu^{1/2}\,dvdS(x)\\
		&=\int_{\Gamma_i}\int_{\R^3}(-v\cdot n(x)){\partial^\alpha f}(t,x,v)(|v|^2-5)\sum_{j\neq i}\big(v_j\pa_{x_j}\phi_c(t,x)\big)\mu^{1/2}\,dvdS(x) = 0,
	\end{align*}where we used \eqref{115q}, \eqref{115c}, \eqref{115b} and \eqref{115d}. Note that $R_x$ maps $v_j$ to $v_j$ for $j\neq i$. 

If $\alpha_i = 1$, then $\alpha_k=0$ for $k\neq i$. Using boundary condition \eqref{120}, we have $\partial_{x_j}\phi_c=0$ on $\Gamma_i$ for any $j=1,2,3$, $j\neq i$. Applying change of variable $v\mapsto R_xv$ to \eqref{20} and using \eqref{115} and \eqref{115b}, we obtain 
		\begin{align*}
		&\quad\,\int_{\Gamma_i}\int_{\R^3}v\cdot n(x){\partial^\alpha f}(t,x,v)(|v|^2-5)v_i\pa_{x_i}\phi_c(t,x)\mu^{1/2}\,dvdS(x)\\
		&\quad\,\int_{\Gamma_i}\int_{\R^3}R_xv\cdot n(x){\partial^\alpha f}(t,x,R_xv)(|R_xv|^2-5)(R_xv)_i\pa_{x_i}\phi_c(t,x)\mu^{1/2}\,dvdS(x)\\
		&\quad\,\int_{\Gamma_i}\int_{\R^3}v\cdot n(x){\partial^\alpha f}(t,x,v)(|v|^2-5)(-v_i)\pa_{x_i}\phi_c(t,x)\mu^{1/2}\,dvdS(x) = 0. 
	\end{align*}
Notice that the above estimates are valid for $i=1,2,3$. Then we obtain 
\begin{align*}
	S_4= 0. 
\end{align*}
Combining the above estimates for $S_j$$(1\le j\le 4)$, taking summation over $|\alpha|\le 2$ of \eqref{100} and letting $\eta$ suitably small, we obtain
\begin{multline}\label{122a}
	\partial_t\sum_{|\alpha|\le 2}(\partial^\alpha f,\Phi_c)_{L^2_{x,v}} + \lambda\sum_{|\alpha|\le 2}\|{\partial^\alpha c}\|^2_{L^2_{x}} \\
	\lesssim \eta \sum_{1\le|\alpha|\le2}\|{\partial^\alpha b}\|^2_{L^2_{x}}
	+C_\eta\sum_{|\alpha|\le 2}\|\{\I-\P\}{\partial^\alpha f}\|^2_{L^2_{x}L^2_D} 
	+C_\eta\sum_{|\alpha|\le 2}\|({\partial^\alpha g},\zeta)_{L^2_v}\|^2_{L^2_{x}}.
\end{multline}for some $\lambda>0$.

		\medskip \noindent{\bf Estimate of ${b}(t,x)$.}
		Now we consider the estimate of ${b}$. For this purpose we choose 
		\begin{align*}
			{\Phi}={\Phi_b}=\sum^3_{m=1}{\Phi^{j,m}_b},\ j=1,2,3,
		\end{align*}
		where 
		\begin{equation*}
			{\Phi^{j,m}_b}=\left\{\begin{aligned}
				\big(|v|^2v_mv_j{\partial_{x_m}\phi_j}-\frac{7}{2}(v_m^2-1){\partial_{x_j}\phi_j}\big)\mu^{1/2},\ m\neq j,\\
				\frac{7}{2}(v_j^2-1){\partial_{x_j}\phi_j}\mu^{1/2},\qquad\qquad\qquad m=j,
			\end{aligned}\right.
		\end{equation*}
	and $\phi_j$($1\le j\le 3$) solves  
				\begin{equation}\label{120a}\left\{\begin{aligned}
						&-\Delta_x \phi_j = {\partial^\alpha b_j},\\
						&{\phi_k}(x) =\partial_n\phi_m(x)= 0 \ \text{ on }\ x\in \Gamma_m,\ \text{ for }k\neq m, \text{ if }\alpha_m = 1,\\
						&\phi_m(x)={\partial_n\phi_k}(x)= 0\ \text{ on }\ x\in \Gamma_m,\ \text{ for }k\neq m, \text{ if }\alpha_m = 0\text{ or } 2.
			\end{aligned}\right.
		\end{equation}
	The existence of solution to \eqref{120a} is guaranteed by \cite[Lamma 4.4.3.1]{Grisvard1985} and we will explain the conditions for pure Neumann type and mixed Dirichlet-Neumann type later. 
	By using the boundary value of $\phi_j$, similar to \eqref{3.8}, we have 
	\begin{equation}\label{3.18}\begin{aligned}
			\sum_{i,k=1}^3\|\partial_{x_ix_k}{\phi_j}\|_{L^2_{x}}^2& = \|\Delta_x\phi_j\|_{L^2_x}^2 \lesssim \|\partial^\alpha b_j\|^2_{L^2_x}.
		\end{aligned} 	
	\end{equation}
Then $S_2$ can be estimated as 
\begin{align}\label{3.19a}
	|S_2|\lesssim \|\partial^\alpha\{\I-\P\}f\|_{L^2_xL^2_D}\sum_{i,j,m=1}^3\|\partial_{x_ix_m}\phi_j\|_{L^2_x}
	\lesssim C_\eta\|\partial^\alpha\{\I-\P\}f\|_{L^2_xL^2_D}+\eta \|\partial^\alpha b\|^2_{L^2_x}.
\end{align}
	Next we fix $1\le j\le 3$ and discuss the value of $|\alpha|$ as before.
	
	If $|\alpha|=0$, then $\phi_j=0$ on $\Gamma_j$ for $j=1,2,3$. Then by \cite[Theorem 6.7-5]{Ciarlet2013}, we have 
	\begin{align*}
		\|\phi_j\|_{L^2_x}\lesssim \|\na_x\phi_j\|_{L^2_x}. 
	\end{align*}
	Similar to \eqref{3.18a}, we can apply the standard elliptic estimate to obtain 
	\begin{align}\label{3.24}
		\|\na_x\phi_j\|_{L^2_x}\lesssim \|b_j\|_{L^2_x},
	\end{align}
and 
\begin{align}\label{3.24a}
	\|\pa_t\na_x\phi_j\|_{L^2_x}\lesssim \|\pa_tb_j\|_{L^2_x}.
\end{align}

	If $|\alpha|=1$, then $\alpha_i=1$ for some $1\le i\le 3$ and $\alpha_k=0$ for $k\neq i$. 
		In particular, if $j=i$, then $\partial_{x_i}\phi_i=0$ on $\Gamma_i$ and $\partial_{x_k}\phi_i=0$ on $\Gamma_k$ for $k\neq i$. In this case, \eqref{120a} is a pure Neumann boundary problem and we need $\int_{\Omega}\partial_{x_i}b_i\,dx=\int_{\Gamma_i}b_i\,dS(x)=0$ to ensure the existence for \eqref{120a}, which follows from \eqref{Lem25a}. Moreover, $\pa_{x_m}\phi_i=0$ on a subset of boundary $\partial\Omega$ with non-zero spherical measure for any $m=1,2,3$. 
	By \cite[Theorem 6.7-5]{Ciarlet2013}, we have 
	\begin{align}\label{3.21a}
		\|\partial_t\partial_{x_m}\phi_i\|_{L^2_x}\lesssim \|\partial_t\na_x\partial_{x_m}\phi_i\|_{L^2_x}
		\lesssim \|\partial_t\partial^\alpha b_i\|_{L^2_x},
	\end{align} and 
	\begin{align}\label{3.21}
		\|\partial_{x_m}\phi_i\|_{L^2_x}\lesssim \|\na_x\partial_{x_m}\phi_i\|_{L^2_x}\lesssim \|\partial^\alpha b_i\|_{L^2_x},
	\end{align}for any $m=1,2,3$, where we used \eqref{3.18} in the second inequalities.

	If $j\neq i$, then $\phi_j=0$ on $\Gamma_i$ and $\Gamma_j$ while $\partial_{x_k}\phi_j=0$ on $\Gamma_k$ for $k\neq j,i$. \eqref{120a} is a mixed Dirichlet-Neumann boundary problem. By \cite[Theorem 6.7-5]{Ciarlet2013}, we have $\|\partial_t\phi_j\|_{L^2_x}\lesssim \|\partial_t\na_x\phi_j\|_{L^2_x}$ and $\|\phi_j\|_{L^2_x}\lesssim \|\na_x\phi_j\|_{L^2_x}$. Thus, by standard elliptic estimates for \eqref{120a}, we have 
	\begin{align}\label{3.21b}
		\|\partial_t\na_x\phi_j\|_{L^2_x}\lesssim \|\partial_t\partial^\alpha b_j\|_{L^2_x},
	\end{align}
	and 
	\begin{align}\label{3.21c}
		\|\na_x\phi_j\|_{L^2_x}\lesssim \|\partial^\alpha b_j\|_{L^2_x}.
	\end{align}
	
	Next we assume $|\alpha|=2$ and $\partial^\alpha = \partial_{x_ix_i}$ for some $1\le i\le 3$.
	Then for $j=1,2,3$, $\phi_j=0$ on $\Gamma_j$ and $\partial_{x_k}\phi_j=0$ on $\Gamma_k$ for $k\neq j$.
	Thus \eqref{120a} is a mixed Dirichlet-Neumann boundary problem and by \cite[Theorem 6.7-5]{Ciarlet2013}, we know that $\|\phi_j\|_{L^2_x}\lesssim \|\nabla_x\phi_j\|_{L^2_x}$. Then from \eqref{120a}, we have 
	\begin{align*}
		\|\na_x\phi_j\|_{L^2_x}^2= \int_\Omega\partial_{x_ix_i}b_j\,\phi_j\,dx
		=  - \int_\Omega\partial_{x_i}b_j\,\partial_{x_i}\phi_j\,dx,
	\end{align*}
	where we used $\pa_{x_i}b_j=0$ on $\Gamma_i$ from Corollary \eqref{Lem25} for $j\neq i$ and $\phi_j=0$ on $\Gamma_i$ if $j=i$. Then we have 
	\begin{align}\label{3.27}
		\|\na_x\phi_j\|_{L^2_x}\le \|\partial_{x_i}b_j\|_{L^2_x}. 
	\end{align}
	Similarly, since $\pa_t$ doesn't affect the boundary values, we have  
	\begin{align}\label{3.28}
		\|\partial_t\na_x\phi_j\|_{L^2_x}\lesssim\|\partial_t\partial_{x_i}b_j\|_{L^2_x}. 
	\end{align}
	Now we let $|\alpha|\le 2$. 
	For $S_1$, we have from \eqref{3.24a}, \eqref{3.21a}, \eqref{3.21b}, \eqref{3.28} and \eqref{11}$_2$ that 
	\begin{align}\label{4.31}\notag
		\dis|S_1| &\le  \big(\P \partial^\alpha f,\partial_t\Phi_b\big)_{L^2_{x,v}}+\big(\{\I-\P\} \partial^\alpha f,\partial_t\Phi_b\big)_{L^2_{x,v}}\\[1mm]
		&\notag\lesssim C_\eta\|\partial^\alpha c\|_{L^2_x}^2 + C_\eta\|\{\I-\P\} \partial^\alpha f\|_{L^2_x}^2 + \eta\|\partial_t\na_x b_j\|_{L^2_x}^2\\[1mm]
		&\lesssim  \sum_{|\alpha|\le 2}\|\partial^\alpha c\|_{L^2_x} +C_\eta\sum_{|\alpha|\le 2}\|\partial^\alpha \{\I-\P\}f\|_{L^2_x} +\eta\sum_{1\le|\alpha|\le2}\|\partial^\alpha a\|_{L^2_x}.
	\end{align}
For $S_3$, by \eqref{3.24}, \eqref{3.21}, \eqref{3.21c} and \eqref{3.27}, we have 
\begin{align}\label{3.22a}\notag
	|S_3|
	&\lesssim\notag C_\eta\|\partial^\alpha \{\I-\P\}f\|^2_{L^2_xL^2_D} +C_\eta\|(\partial^\alpha g,\zeta)_{L^2_v}\|^2_{L^2_x}+ \eta\|\nabla_x\phi_j\|_{L^2_x}\\
	&\lesssim C_\eta\sum_{|\alpha|\le 2}\|\{\I-\P\}\partial^\alpha f\|_{L^2_xL^2_D}^2 + C_\eta\sum_{|\alpha|\le 2}\|(\pa^\alpha g,\zeta(v))_{L^2_v}\|_{L^2_x}^2 + \eta\sum_{|\alpha|\le1}\|\partial^\alpha b\|_{L^2_x}.
\end{align}
For the second term on left hand side of \eqref{100}, we have  
		\begin{align}\label{3.30}\notag
			&\quad\,-\sum^3_{m=1}(\P {\partial^\alpha f},v\cdot \na_x{\Phi^{j,m}_b})_{L^2_{x,v}}\\
			&\notag=-\sum^3_{m=1}\big(({\partial^\alpha a}+{\partial^\alpha b}\cdot v+\frac{1}{2}(|v|^2-3){\partial^\alpha c})\mu^{1/2},v\cdot{\nabla_x\Phi^{j,m}_b}\big)_{L^2_{x,v}}\\
			\notag&=-\sum^{3}_{m=1,m\neq j}(v_mv_j\mu^{1/2}{\partial^\alpha b_j},|v|^2v_mv_j\mu^{1/2}{\partial_{x_m}^2\phi_j})_{L^2_{x,v}}\\
			&\notag\qquad-\sum^{3}_{m=1,m\neq j}(v_mv_j\mu^{1/2}{\partial^\alpha b_m},|v|^2v_mv_j\mu^{1/2}{\partial_{x_m}\partial_{x_j}\phi_j})_{L^2_{x,v}}\\
			&\notag\qquad+7\sum^{3}_{m=1,m\neq j}({\partial^\alpha b_m},{\partial_{x_m}\partial_{x_j}\phi_j})_{L^2_{x}}-7({\partial^\alpha b_m},{\partial^2_{x_j}\phi_j})_{L^2_{x}}\\
			&= -7 \sum^3_{m=1}({\partial^\alpha b_j},{\partial_{x_m}^2\phi_j})_{L^2_{x}}=7\|{\partial^\alpha b_j}\|^2_{L^2_{x}}.
		\end{align}
%
%
%
%
		Now we consider the boundary term $S_4$. As in the estimate on $c(t,x)$, we consider $\Gamma_i$ for fixed $i=1,2,3$: 
		\begin{multline}\label{119}
			\int_{\Gamma_i}(v\cdot n(x){\partial^\alpha f}(x),{\Phi_b}(x))_{L^2_v}\,dS(x)\\
			= \sum_{m=1}^3\int_{\Gamma_i}\int_{\R^3}v\cdot n(x){\partial^\alpha f}(t,x,v){\Phi^{j,m}_b}(x,v)\,dvdS(x).
		\end{multline}
		If $\alpha_i=0$ or $2$, then applying boundary condition \eqref{120a}, we have that for $x\in\Gamma_i$, 
		\begin{align*}
			\partial_{x_i}\phi_j(x) = \partial_{x_j}\phi_i(x) = 0,\quad\text{ for }j\neq i. 
		\end{align*}
	This shows that $\Phi^{j,m}_b(x,v)$ is even with respect to $v_i$ when $x\in\Gamma_i$. 
	Noticing $R_xv = v-2v\cdot e_ie_i$ on $\Gamma_i$, we know that on $\Gamma_i$, 
	\begin{align*}
		{\Phi^{j,m}_b}(x,R_xv) = {\Phi^{j,m}_b}(x,v).
	\end{align*} 		
	Applying change of variable $v\mapsto R_xv$ and using identities \eqref{115q}, \eqref{115c}, \eqref{115b} and \eqref{115d}, \eqref{119} becomes 
		\begin{align*}
			&\quad\, \sum_{m=1}^3\int_{\Gamma_i}\int_{\R^3}v\cdot n(x){\partial^\alpha f}(t,x,v){\Phi^{j,m}_b}(x,v)\,dvdS(x)\\
			&=\sum_{m=1}^3\int_{\Gamma_i}\int_{\R^3}R_xv\cdot n(x){\partial^\alpha f}(t,x,R_xv){\Phi^{j,m}_b}(x,R_xv)\,dvdS(x)\\
			&=\sum_{m=1}^3\int_{\Gamma_i}\int_{\R^3}-v\cdot n(x){\partial^\alpha f}(t,x,v){\Phi^{j,m}_b}(x,v)\,dvdS(x)=0.
		\end{align*}
	If $\alpha_i=1$, then boundary condition \eqref{120a} shows that on $x\in\Gamma_i$, 
	\begin{align*}
		\partial_{x_j}\phi_j(x) &= 0, \quad \text{ for }j=1,2,3,\\
		\partial_{x_m}\phi_j(x) &= 0,\quad \text{ for } j,m\neq i.
	\end{align*}
Then we know that $\Phi^{j,m}_b(x,v)$ is odd with respect to $v_i$ when $x\in\Gamma_i$ and hence, 
\begin{align*}
	{\Phi^{j,m}_b}(x,R_xv) = -{\Phi^{j,m}_b}(x,v).
\end{align*}
Now applying change of variable $v\mapsto R_xv$ and using identities \eqref{115} and \eqref{115b}, \eqref{119} becomes 
\begin{align*}
	&\quad\, \sum_{m=1}^3\int_{\Gamma_i}\int_{\R^3}v\cdot n(x){\partial^\alpha f}(t,x,v){\Phi^{j,m}_b}(x,v)\,dvdS(x)\\
	&=\sum_{m=1}^3\int_{\Gamma_i}\int_{\R^3}R_xv\cdot n(x){\partial^\alpha f}(t,x,R_xv){\Phi^{j,m}_b}(x,R_xv)\,dvdS(x)\\
	&=\sum_{m=1}^3\int_{\Gamma_i}\int_{\R^3}v\cdot n(x){\partial^\alpha f}(t,x,v){(-\Phi^{j,m}_b)}(x,v)\,dvdS(x)=0.
\end{align*}
Therefore, \begin{equation}\label{3.31}
	S_4=0.
\end{equation} Combining estimates \eqref{3.19a}, \eqref{4.31}, \eqref{3.22a}, \eqref{3.30}, \eqref{3.31}, taking summation of \eqref{100} over $|\alpha|\le2$ and letting $\eta$ sufficiently small, we have 
		\begin{multline}\label{122b}
			\partial_t\sum_{|\alpha|\le 2}(\partial^\alpha f,\Phi_b)_{L^2_{x,v}} + \lambda\sum_{|\alpha|\le 2}\|{\partial^\alpha b}\|^2_{L^2_{x}}\lesssim \eta\sum_{1\le|\alpha|\le 2}\|\partial^\alpha a\|^2_{L^2_{x}}+ C_\eta\sum_{|\alpha|\le 2}\|\pa^\alpha{c}\|^2_{L^2_{x}}\\
			+C_\eta\sum_{|\alpha|\le 2}\|\pa^\alpha\{\I-\P\}{f}\|^2_{L^2_{x}L^2_D}  +C_\eta\sum_{|\alpha|\le 2}\|({\partial^\alpha g},\zeta)_{L^2_v}\|^2_{L^2_{x}},
		\end{multline}for some $\lambda>0$.

		\medskip\noindent{\bf Estimate on ${a}(t,x)$:} We choose the following test function
		\begin{align*}
			{\Phi} = {\Phi_{a}} = (|v|^2-10)\big(v\cdot{\nabla_{x}\phi_{a}}(t,x)\big)\mu^{1/2},
		\end{align*}
		where $\phi_a$ solves 
		\begin{equation}\label{122}\left\{\begin{aligned}
				&-\Delta_x \phi_a = {\partial^\alpha a},\\
				&{\phi_a}(x)= 0 \ \text{ on }\ x\in \Gamma_i,\ \text{ if }\alpha_i = 1,\\
				&\frac{\partial\phi_a}{\partial n}(x)= 0\ \text{ on }\ x\in \Gamma_i,\ \text{ if }\alpha_i = 0\text{ or } 2.
			\end{aligned}\right.
\end{equation}
	The existence and uniqueness of solution to \eqref{122} is guaranteed by \cite[Lamma 4.4.3.1]{Grisvard1985}. When $\alpha_i=2$ for some $i$, \eqref{122} is pure Neumann problem and we need $\int_{\Omega}\partial_{x_ix_i}a\,dx=\int_{\Gamma_i}\partial_{x_i}a\,dS(x)=0$ from Corollary \ref{Lem25} to ensure the existence of \eqref{122}. 
		Now we compute \eqref{100}. For the second term on left hand side of \eqref{100}, we have 
		\begin{align*}
			&\quad\,-({\partial^\alpha \P f},v\cdot{\nabla_{x}\Phi_{a}})_{L^2_{x,v}} \\
			&= -\sum_{j,m=1}^3({\partial^\alpha a}+{\partial^\alpha b}\cdot v+\frac{1}{2}(|v|^2-3){\partial^\alpha c} ,v_jv_m(|v|^2-10)\mu{\partial_j\partial_m\phi_{a}})_{L^2_{x,v}} \\
			&= \sum_{j=1}^3({\partial^\alpha a} ,{-\partial^2_j\phi_{a}})_{L^2_{x}}  = \|{\partial^\alpha a}\|^2_{L^2_{x}} .
		\end{align*}
		Since $\Phi_a$ and $\Phi_c$ has similar structure, the estimates for $S_j$$(1\le j\le 4)$ are similar to the case of $c(t,x)$ from \eqref{3.8} to \eqref{122a}. In fact, similar to the calculation from \eqref{3.8} to \eqref{3.13}, we have that for $|\alpha|\le 2$, 
		\begin{equation}\label{3.32}
			\sum_{i,j=1}^3\|\partial_{x_ix_j}\phi_a\|^2_{L^2_x}=\|\Delta_x\phi_a\|^2_{L^2_x}\lesssim \|\partial^\alpha a\|^2_{L^2_x},
		\end{equation}
	\begin{equation}\label{3.33}
		\|\nabla_x\phi_a\|_{L^2_x}\lesssim \sum_{|\alpha|\le1}\|\pa^\alpha a\|_{L^2_x},
	\end{equation}
and 
\begin{equation}\label{3.34}
	\|\pa_t\nabla_x\phi_a\|_{L^2_x}\lesssim \sum_{|\alpha|\le1}\|\pa_t\pa^\alpha a\|_{L^2_x}\lesssim \sum_{1\le|\alpha|\le2}\|\pa^\alpha b\|_{L^2_x}.
\end{equation}
where the last inequality follows from \eqref{11}$_1$. 
	Then for $S_1$, we apply \eqref{3.34} to obtain 
	\begin{align*}
		|S_1|&\le \big|\big(\{\I-\P\}\partial^\alpha f, \pa_t\Phi_a\big)_{L^2_{x,v}}\big|
		+\big|\big(\P\partial^\alpha f, \pa_t\Phi_a\big)_{L^2_{x,v}}\big|\\
		&\lesssim \|\{\I-\P\}\partial^\alpha f\|^2_{L^2_xL^2_D} + \|\partial^\alpha b\|_{L^2_x}^2 + \|\pa_t\na_x\phi_a\|^2_{L^2_x}\\
		&\lesssim \|\{\I-\P\}\partial^\alpha f\|^2_{L^2_xL^2_D} + \sum_{|\alpha|\le 2}\|\partial^\alpha b\|_{L^2_x}^2. 
	\end{align*}	
For $S_2$, by \eqref{3.32}, we have 
\begin{align*}
	|S_2|\lesssim C_\eta\|\{\I-\P\}\partial^\alpha f\|^2_{L^2_xL^2_D} + \eta\|\partial^\alpha a\|^2_{L^2_x}. 
\end{align*}
For $S_3$, by \eqref{3.33}, we have 
\begin{align*}
	|S_3|\lesssim C_\eta \|\{\I-\P\}\partial^\alpha f\|^2_{L^2_xL^2_D} + C_\eta \|(\partial^\alpha g,\zeta)_{L^2_v}\|^2_{L^2_x} + \eta\sum_{|\alpha|\le1} \|\pa^\alpha a\|_{L^2_x}^2. 
\end{align*}
	For $S_4$, we will apply a similar argument as in the estimate of $c(t,x)$ to calculate the boundary value. As before, we decompose $\partial\Omega=\cap^{3}_{i=1}\Gamma_i$ and calculate the value on each boundary separately:
	\begin{align}\label{123}
		\int_{\Gamma_i}(v\cdot n(x){\partial^\alpha f}(x),{\Phi_a}(x))_{L^2_v}\,dS(x).
	\end{align}
If $\alpha_i=0$ or $2$, using boundary value \eqref{122}, we know that $\partial_{x_i}\phi_a(x)=0$ on $x\in\Gamma_i$ and hence, $\Phi_a(t,x,v)$ is even with respect to $v_i$. Also, $R_x$ maps $v$ to $v-2v\cdot e_je_j$ on $\Gamma_i$. Thus,  on $\Gamma_i$, we have 
\begin{align*}
	\Phi_a(t,x,R_xv) = \Phi_a(t,x,v).
\end{align*} 
Therefore, applying change of variable $v\mapsto R_xv$, \eqref{123} becomes 
\begin{align*}
	&\quad\,\int_{\Gamma_i}(v\cdot n(x){\partial^\alpha f}(x),{\Phi_a}(x))_{L^2_v}\,dS(x)\\
	&=\int_{\Gamma_i}\int_{\R^3}R_xv\cdot n(x){\partial^\alpha f}(x,R_xv){\Phi_a}(x,R_xv))\,dvdS(x)\\
	&= \int_{\Gamma_i}\int_{\R^3}-v\cdot n(x){\partial^\alpha f}(x,v){\Phi_a}(x,v)\,dvdS(x)=0, 
\end{align*}
where we also used \eqref{115q}, \eqref{115c}, \eqref{115b} and \eqref{115d}. 
If $\alpha_i=1$, by \eqref{122}, one has on $\Gamma_i$ that 
\begin{align*}
	\Phi_a(t,x,R_xv) = -\Phi_a(t,x,v).
\end{align*} 
Therefore, applying change of variable $v\mapsto R_xv$ and using \eqref{115} and \eqref{115b}, \eqref{123} becomes 
\begin{align*}
	&\quad\,\int_{\Gamma_i}(v\cdot n(x){\partial^\alpha f}(x),{\Phi_a}(x))_{L^2_v}\,dS(x)\\
	&=\int_{\Gamma_i}\int_{\R^3}R_xv\cdot n(x){\partial^\alpha f}(x,R_xv){\Phi_a}(x,R_xv))\,dvdS(x)\\
	&= \int_{\Gamma_i}\int_{\R^3}v\cdot n(x){\partial^\alpha f}(x,v)({-\Phi_a}(x,v))\,dvdS(x)=0. 
\end{align*}
In any cases, we have $S_4=0$. 
Combining the above estimates, taking summation over $|\alpha|\le2$ of \eqref{100} and letting $\eta>0$ small enough, we have 
\begin{multline}\label{122d}
\partial_t\sum_{|\alpha|\le2}(\partial^\alpha f,\Phi_a)_{L^2_{x,v}} + \lambda\sum_{|\alpha|\le2}\|{\partial^\alpha a}\|^2_{L^2_{x}}\\
\lesssim\sum_{|\alpha|\le 2} \|\partial^\alpha\{\I-\P\}{f}\|^2_{L^2_{x}L^2_D}+\sum_{|\alpha|\le 2}\|{\partial^\alpha b}\|^2_{L^2_{x}} +\sum_{|\alpha|\le 2}\|({\partial^\alpha g},\zeta)_{L^2_v}\|^2_{L^2_{x}}.
\end{multline}

Now we take the linear combination $\eqref{122a}+\kappa\times\eqref{122b}+\kappa^2\times\eqref{122d}$, summation on $1\le|\alpha|\le 2$ and let $\kappa,\eta$ sufficiently small, then 
		\begin{equation*}
			\partial_t\E_{int}(t) + \lambda\sum_{|\alpha|\le 2}\|\pa^\alpha[{a},{b},{c}]\|^2_{L^2_{x}}
			\lesssim \sum_{|\alpha|\le 2}\|({\partial^\alpha g},\zeta)_{L^2_v}\|^2_{L^2_{x}}+\sum_{|\alpha|\le 2}\|\partial^\alpha\{\I-\P\}{ f}\|_{L^2_{x}L^2_D},
		\end{equation*}
	where we used \eqref{2.14} and $\E_{int}(t)$ is given by 
	\begin{align*}
		\E_{int}(t) = \sum\Big((\partial^\alpha f,\Phi_c)_{L^2_{x,v}} + \kappa(\partial^\alpha f,\Phi_b)_{L^2_{x,v}}+\kappa^2(\partial^\alpha f,\Phi_a)_{L^2_{x,v}}\Big),
	\end{align*}where the summation is taken over ${1\le|\alpha|\le 2}$ and we restrict $\alpha_i=2$ for some $i$ when $|\alpha|=2$. 
Using \eqref{3.10}, \eqref{3.12}, \eqref{3.21}, \eqref{3.27} and \eqref{3.33}, we know that 
\begin{align*}
	\E_{int}(t)\lesssim \sum_{|\alpha|\le 2}\|\partial^\alpha f\|_{L^2_xL^2_v}. 
\end{align*}This completes the Theorem \ref{Thm31}.
\end{proof}
		
Now we estimate $\|({\partial^\alpha g},\zeta)_{L^2_v}\|^2_{L^2_{x}}$ when $g=\Gamma(f,f)$. For $1\le|\alpha|\le 2$, by \eqref{gamma}, we have 
		\begin{align}
			\label{135}
			\int_{\Omega}|({\partial^\alpha\Gamma(f,f)},\zeta(v))_{L^2_{v}}|^2\,dx
			&\notag\lesssim \int_{\Omega}\sum_{\alpha_1\le\alpha}|{\partial^{\alpha-\alpha_1} f}|^2_{L^2_v}|{\partial^{\alpha_1} f}|^2_{L^2_{D}}\,dx\\
			&\notag\lesssim \sum_{|\alpha_1|=2}\|{\partial^{\alpha-\alpha_1} f}\|^2_{L^\infty_xL^2_v}\|{\partial^{\alpha_1} f}\|^2_{L^2_xL^2_{D}}\\
			&\notag\qquad+\sum_{|\alpha_1|=1}\|{\partial^{\alpha-\alpha_1} f}\|^2_{L^3_xL^2_v}\|{\partial^{\alpha_1} f}\|^2_{L^6_xL^2_{D}}\\
			&\notag\qquad+\sum_{|\alpha_1|=0}\|{\partial^{\alpha-\alpha_1} f}\|^2_{L^2_xL^2_v}\|{\partial^{\alpha_1} f}\|^2_{L^\infty_xL^2_{D}}\\
			&\lesssim \|{ f}\|^2_{H^2_xL^2_v}\| {f}\|^2_{L^2_xL^2_{D}}\lesssim \E(t)\D(t). 
		\end{align}where we used embedding $\|f\|_{L^3_x(\Omega)}\lesssim\|f\|_{H^1_x(\Omega)}$, $\|f\|_{L^6_x(\Omega)}\lesssim\|\na_xf\|_{L^2_x(\Omega)}$ and $\|f\|_{L^\infty_x(\Omega)}\lesssim\|f\|_{H^2_x(\Omega)}$ from \cite[Section V]{Adams2003}.

\section{Global existence}\label{Sec4}
In this section, we will prove the main Theorem \ref{Main}.

\begin{proof}[Proof of Theorem \ref{Main}]
	Let $|\alpha|\le 2$ and apply $\partial^\alpha$ to \eqref{1}, we have 
	\begin{align}\label{24}
	\partial_t{\partial^\alpha f}+v\cdot\nabla_x{\partial^\alpha f} - L {\partial^\alpha f} = \partial^\alpha\Gamma(f,f).
	\end{align}
	Taking inner product of \eqref{24} with $\partial^\alpha f$ over $\Omega\times\R^3$, we have 
	\begin{multline}\label{25}
		\frac{1}{2}\partial_t\|\partial^\alpha f\|^2_{L^2_{x,v}} + \frac{1}{2}\int_{\partial\Omega}\int_{\R^3}v\cdot n(x)|\partial^\alpha f(x,v)|^2\,dvdS(x) + \lambda\|\{\I-\P\}\partial^\alpha f\|_{L^2_{x}L^2_D}^2\\ \lesssim \|f\|_{H^2_xL^2_v}\|f\|_{H^2_xL^2_D}\|\partial^\alpha\{\I-\P\} f\|_{L^2_xL^2_D},
	\end{multline}where we used \eqref{L} and \eqref{gammax} and $dS(x)$ is the spherical measure.
By \eqref{115q}, \eqref{115c}, \eqref{115} and \eqref{115d}, we know that on interior of each $\Gamma_i(1\le i\le 3)$, 
\begin{align*}
|\partial^\alpha f(x,R_xv)|^2 = |\partial^\alpha f(x,v)|^2,\quad\text{ on $v\cdot n(x)\neq 0$}. 
\end{align*}
Then by changing of variable $v\mapsto R_xv$, we have 
\begin{align*}
	&\quad\,\int_{\Omega}\int_{\R^3}v\cdot n(x)|\partial^\alpha f(x,v)|^2\,dvdS(x) \\
	&= \int_{\Omega}\int_{\R^3}R_xv\cdot n(x)|\partial^\alpha f(x,R_xv)|^2\,dvdS(x)\\
	&= \int_{\Omega}\int_{\R^3}-v\cdot n(x)|\partial^\alpha f(x,v)|^2\,dvdS(x) = 0. 
\end{align*}
Therefore, taking summation on $|\alpha|\le 2$ of \eqref{25}, we have 
\begin{equation}\label{9}
	\frac{1}{2}\partial_t\|f\|^2_{H^2_xL^2_v} + \lambda\|\{\I-\P\}f\|_{H^2_xL^2_D}^2 \lesssim \|f\|_{H^2_xL^2_v}\|f\|_{H^2_xL^2_D}\|\{\I-\P\} f\|_{H^2_xL^2_D}. 
\end{equation}
Now we take linear combination $\eqref{9}+\kappa\times\eqref{121}$ with $\kappa>0$ small enough and apply \eqref{135}, then 
\begin{align}\label{9b}
	\partial_t\E(t) + \lambda \D(t) \lesssim (\sqrt{\E(t)}+\E(t))\D(t), 
\end{align}
 where $\E(t)$ is given by 
 \begin{align*}
 	\E(t) := \frac{1}{2}\|f(t)\|^2_{H^2_xL^2_v} + \kappa\E_{int}(t),
 \end{align*}
and $\D(t)$ is defined by \eqref{defd}. It's direct to check that $\E(t)$ satisfies \eqref{defe} with $\kappa>0$ sufficiently small. 
With the main estimate \eqref{9b} in hand, it is now standard to apply the continuity argument and local existence from Section \ref{Sec5} to prove the global-in-time existence and uniqueness of \eqref{1} and \eqref{specular}, under the smallness of \eqref{small}.


For large-time decay, when $\gamma+2s\ge 0$, we have $\E(t)\lesssim \D(t)$. Then under the smallness \eqref{small} of $\E(0)$, it's standard to apply the {\em a priori} estimate argument to obtain 
\begin{equation*}
	\partial_t\E(t) + \delta \E(t) \le 0,
\end{equation*}
and 
	\begin{equation*}
	\E(t) \le e^{-\delta t}\|f_0\|^2_{H^2_xL^2_v},
\end{equation*}for some generic constant $\delta>0$. 
This gives the large-time behavior for {\em hard} potential. 

For {\em soft} potential, we will make use of additional velocity weight and calculate the weighted estimates first. 
Taking inner product of \eqref{24} with $w^2\partial^\alpha f$ over $\Omega\times\R^3$ and summation over $|\alpha|\le2$, we have from Lemma \ref{L2} and \ref{gam} that 
\begin{align}\label{9a}\notag
	&\quad\,\frac{1}{2}\partial_t\sum_{|\alpha|\le 2}\|w\pa^\alpha f\|^2_{L^2_xL^2_v} + \lambda\sum_{|\alpha|\le 2}\|w\partial^\alpha f\|_{L^2_xL^2_D}^2 \\
	&\notag\lesssim \sum_{|\alpha|\le 2}\|\pa^\alpha f\|_{L^2_xL^2_{B_C}}^2 + \|wf\|_{H^2_xL^2_v}\|f\|_{H^2_xL^2_{D,w}}\| f\|_{H^2_xL^2_{D,w}}\\
	&\lesssim \D(t) + \sqrt{\E_w(t)}\D_w(t),
\end{align}for some $\lambda>0$. 
Notice that $w(v)=w(R_xv)$ on $x\in\partial\Omega$. So, $w$ doesn't affect the vanishing boundary term. 
Taking linear combination $\eqref{9b}+\kappa\times\eqref{9a}$ 
with $\kappa>0$ small enough, we have 
\begin{align*}
	\partial_t\E_w(t) + \lambda \D_w(t) \lesssim (\sqrt{\E(t)}+\E(t))\D(t) + \sqrt{\E_w(t)}\D_w(t), 
\end{align*}
where $\E_w$ is given by 
\begin{align*}
	\E_w(t) := \frac{1}{2}\|f(t)\|^2_{H^2_xL^2_v} + \kappa\E_{int}(t) + \kappa\frac{1}{2}\sum_{|\alpha|\le 2}\|w\pa^\alpha f\|^2_{L^2_xL^2_v},
\end{align*}
and $\D_w$ is defined by \eqref{defdw}. It's direct to check that $\E_w$ satisfies \eqref{defew}.
Since $\E(t)\lesssim \E_w(t)$ and $\D(t)\lesssim \D_w(t)$, under the smallness assumption \eqref{small2}, we can obtain the closed estimate:
\begin{align*}
	\partial_t\E_w(t) + \lambda \D_w(t) \le 0. 
\end{align*}
Taking integration on $t\in[0,T]$ for any $T\in(0,\infty]$, we have 
\begin{align}\label{28}
	\sup_{0\le t\le T}\|wf\|^2_{H^2_xL^2_v} + \int^T_0\|w\partial^\alpha f\|_{H^2_xL^2_D}^2\,dt\le \|wf_0\|^2_{H^2_xL^2_v}.
\end{align}

Now we are ready to prove the large-time behavior for {\em soft} potential $\gamma+2s<0$. Let 
\begin{align*}
	h = e^{\delta t^p}f,
\end{align*}
with $\delta>0$, $0<p<1$ chosen later. Since $f$ solves \eqref{24}, we know that $h$ solves 
\begin{align*}
	\partial_t{\partial^\alpha h}+v\cdot\nabla_x{\partial^\alpha h} - L {\partial^\alpha h} = e^{-\delta t}\partial^\alpha\Gamma(h,h) + \delta pt^{p-1}\partial^\alpha h,\quad h|_{t=0} = f_0. 
\end{align*}
Taking inner product with $h$, following the same argument for deriving \eqref{9}, we have 
\begin{align}\label{21}
	\frac{1}{2}\partial_t\|h\|^2_{H^2_xL^2_v} + \lambda\|\{\I-\P\}h\|_{H^2_xL^2_D}^2 \lesssim \|h\|_{H^2_xL^2_v}\|h\|_{H^2_xL^2_D}\|h\|_{H^2_xL^2_D} + \delta pt^{p-1}\|h\|^2_{H^2_xL^2_v}. 
\end{align}
Following the same argument for deriving \eqref{121} and using \eqref{135}, there exists $\E_{int,2}$ satisfying 
\begin{align}\label{4.9}
	|\E_{int,2}|\lesssim \|h\|^2_{H^2_xL^2_v},
\end{align} such that 
\begin{equation}\label{22}
	\partial_t\E_{int,2}(t) + \lambda\|e^{\delta t}[{a},{b},{c}]\|^2_{H^2_{x}}
	\lesssim \|h\|^2_{H^2_xL^2_v}\|h\|_{H^2_xL^2_D}^2+\|\{\I-\P\}{h}\|^2_{H^2_{x}L^2_D}+ \delta pt^{p-1}\|h\|^2_{H^2_xL^2_v}.
\end{equation}
Taking linear combination $\eqref{21}+\kappa\times \eqref{22}$ with $\kappa>0$ small enough, we have 
\begin{align*}
	\partial_t\E_2(t) + \lambda\|h\|_{H^2_xL^2_D}^2 \lesssim \|h\|_{H^2_xL^2_v}\|h\|^2_{H^2_xL^2_D}+\|h\|^2_{H^2_xL^2_v}\|h\|_{H^2_xL^2_D}^2 + \delta pt^{p-1}\|h\|^2_{H^2_xL^2_v},
\end{align*}
where $\E_2(t)$ is given by 
\begin{align*}
	\E_2(t) = \frac{1}{2}\|h\|^2_{H^2_xL^2_v}+\kappa\E_{int,2}(t).
\end{align*}
From \eqref{4.9}, we know that $\E_2(t)\approx \|h\|^2_{H^2_xL^2_x}$. 
Under the smallness of $\|h|_{t=0}\|^2_{H^2_xL^2_v} = \|f_0\|^2_{H^2_xL^2_v}$, it's standard to apply the {\em a priori} estimate argument to obtain
\begin{equation*}
	\partial_t\E_2(t) + \lambda\|h\|_{H^2_xL^2_D}^2 \lesssim \delta pt^{p-1}\|h\|^2_{H^2_xL^2_v}.
\end{equation*}
Taking integration over $t\in[0,T]$ for any $T\in(0,\infty]$, we have 
\begin{equation}\label{26}
	\sup_{0\le t\le T}\|h\|^2_{H^2_xL^2_v} + \lambda\int^T_0\|h\|_{H^2_xL^2_D}^2\,dt \lesssim\|f_0\|^2_{H^2_xL^2_v}+ \delta p\int^T_0t^{p-1}\|h\|^2_{H^2_xL^2_v}\,dt. 
\end{equation}
As in \cite{Strain2012, Duan2020}, for $p'>0$ to be chosen depend on $p$, we define 
\begin{equation*}
	\mathbf{E} = \{\<v\>\le  t^{p'}\}, \quad 	\mathbf{E}^c = \{\<v\>>  t^{p'}\},
\end{equation*}
and make decomposition $\1=\1_{\mathbf{E}}+\1_{\mathbf{E}^c}$. Then the second right-hand term of \eqref{26} can be bounded by 
\begin{equation*}
	\delta p\int^T_0\int_{\R^3}t^{p-1}\1_{\mathbf{E}}\|h\|^2_{H^2_x}\,dvdt + \delta p\int^T_0\int_{\R^3}t^{p-1}\1_{\mathbf{E}^c}\|h\|^2_{H^2_x}\,dvdt =:I_1+I_2. 
\end{equation*}
We define $p' = \frac{p-1}{\gamma+2s}$ for Boltzmann case and $p'=\frac{p-1}{\gamma+2}$ for Landau case. Then on $\mathbf{E}$, we have 
\begin{align*}
	t^{p-1} \le \<v\>^{\frac{p-1}{p'}}, 
\end{align*}and hence, 
\begin{equation}\label{37}
	I_1 \le \delta p\int^T_0\|h\|^2_{H^2_xL^2_D}\,dt.
\end{equation}
On the other hand, on $\mathbf{E}^c$, we have 
\begin{align*}
	w^{-2} \le e^{-\frac{q\<v\>^\vt}{2}} \le e^{-\frac{q t^{p'\vt}}{2}}. 
\end{align*}
Choosing $p = p'\vt$, i.e. $p$ satisfies \eqref{p}, and $2\delta<q/2$, we have 
\begin{align}\label{38}\notag
	I_2 &\le \delta p\int^T_0\int_{\R^3}t^{p-1}e^{-\frac{q t^{p'\vt}}{2}}\|e^{\delta t^p}wf\|^2_{H^2_x}\,dvdt\\
	&\notag\le \delta p\int^T_0t^{p-1}e^{-\frac{q t^{p'\vt}}{2}}e^{2\delta t^p}\,dt\ \sup_{0\le t\le T}\|wf\|^2_{H^2_xL^2_v}\\
	&\le C\delta p\sup_{0\le t\le T}\|wf\|^2_{H^2_xL^2_v}. 
\end{align}
Substituting \eqref{37} and \eqref{38} into \eqref{26}, we have 
\begin{align*}
	\sup_{0\le t\le T}\|h\|^2_{H^2_xL^2_v} + \lambda\int^T_0\|h\|_{H^2_xL^2_D}^2\,dt \lesssim\|f_0\|^2_{H^2_xL^2_v}+ \delta p\int^T_0\|h\|^2_{H^2_xL^2_D}\,dt + C\delta p\sup_{0\le t\le T}\|wf\|^2_{H^2_xL^2_v}. 
\end{align*}
Then applying \eqref{28} to control the last term and letting $\delta>0$ small enough, we have 
\begin{equation*}
	\sup_{0\le t\le T}\|h\|^2_{H^2_xL^2_v} + \lambda\int^T_0\|h\|_{H^2_xL^2_D}^2\,dt\lesssim\|wf_0\|^2_{H^2_xL^2_v}. 
\end{equation*}
This implies the time-decay estimate \eqref{timedecay} for soft potential case. 
Then we complete the proof of Theorem \ref{Main}. 
\end{proof}

\section{Local existence}\label{Sec5}

This section is devoted to the local existence for equation \eqref{1} with specular boundary condition \eqref{specular}. With the {\em a priori} estimate in Section \ref{Sec4} and local-in-time existence, we are able to close the proof of global-in-time existence. 
\begin{Thm}\label{localsolution}
	Assume $\gamma\ge -3$ for Landau case and $\gamma>\max\{-2s-\frac{3}{2},-3\}$ for Boltzmann case. 
	Then there exists $\ve_0>0$, $T_0>0$ such that if $F_0(x,v)=\mu+\mu^{1/2}f_0(x,v)\ge 0$ and 
	\begin{align*}
		\|{wf_0}\|_{H^2_xL^2_v}\le\ve_0,
	\end{align*}
	then the specular reflection boundary problem \eqref{1} and \eqref{specular} admits a unique solution $f(t,x,v)$ on $t\in[0,T_0]$, $x\in\Omega$, $v\in\R^3$, satisfying the uniform estimate 
	\begin{align}
		\label{169}
		\sup_{0\le t\le T_0}\|wf\|_{H^2_xL^2_v}+\int^{T_0}_0\|wf\|_{H^2_xL^2_D}^2\,dt\lesssim\|wf_0\|^2_{H^2_xL^2_v}.
	\end{align}
\end{Thm}

We begin with the following linear inhomogeneous problem:
\begin{equation}\label{154}\left\{
	\begin{aligned}
		&\partial_tf + v\cdot \nabla_xf - A f = \Gamma(h,f)+Kh,\\
		&f(0,x,v) = f_0(x,v),\\ 
		&f(t,x,v) = f(t,x,R_xv), \quad \text{ on $\gamma_-$},
	\end{aligned}\right.
\end{equation}
for a given function $h=h(t,x,v)$, where $A$ and $K$ and defined by \eqref{AK1} for Landau case and \eqref{AK} for Boltzmann case. Then we have the following Lemma on existence of linear equation \eqref{154}. 

\begin{Lem}\label{Lem22}
	There exists $\varepsilon_0>0$, $T_0>0$ such if 
	\begin{align*}
		wf_0\in H^2_xL^2_v,\quad wh\in L^\infty_{T_0}L^2_{x,v}\cap L^2_{T_0}H^2_xL^2_{D},
	\end{align*}
	satisfying 
	\begin{equation*}
		h(t,x,R_xv)=h(t,x,v),\ \partial_nh(t,x,R_xv)=-\partial_nh(t,x,v)\text{ on }v\cdot n(x)\neq 0,
	\end{equation*} and 
	\begin{align}\label{156a}
		\|wf_0\|_{H^2_xL^2_v}+\|wh\|_{L^\infty_{T_0}H^2_xL^2_v}+\|wh\|_{L^2_{T_0}H^2_{x}L^2_{D}}\le \varepsilon_0,
	\end{align}
	then the initial boundary value problem \eqref{154} admits a unique weak solution $f=f(t,x,v)$ on $[0,T_0]\times\Omega\times\R^3$ satisfying
		\begin{equation}\label{411}
		\partial_nf(t,x,R_xv)=-\partial_nf(t,x,v)\text{ on }v\cdot n(x)\neq 0,
	\end{equation}
and 
	\begin{align}\label{158a}
		\|wf\|_{L^\infty_{T_0}H^2_xL^2_v}+\|wf\|_{L^2_{T_0}H^2_{x}L^2_{D}}\le \|wf_0\|_{H^2_xL^2_v}+T_0^{1/2}\|wh\|_{L^\infty_{T_0}H^2_xL^2_v}. 
	\end{align}
\end{Lem}
\begin{proof}
	Let $\eta_v$ and $\eta_x$ be the standard mollifier in $\R^3$ and $\Omega$: $\eta_v,\eta_x\in C^\infty_c$, $0\le \eta_v,\eta_x\le 1$, $\int\zeta_vdv=\int\zeta_xdx=1$. For $\varepsilon>0$, let $\eta^\varepsilon_v(v) = \ve^{-3}\zeta_v{(\ve^{-1}v)}$ and $\eta_x^\ve(x)=\ve^{-3}\zeta_x{(\ve^{-1}x)}$. We mollify the initial data as $f_0^\ve=f_0*\eta^\ve_v*\eta^\ve_x$. Then 
	\begin{align*}
		\|{wf_0^\ve}\|_{H^2_xL^2_v}\le \|wf_0*\eta^\ve_v*\eta^\ve_x\|_{H^2_xL^2_v}\lesssim \|\eta^{\ve}_v\|_{L^1_{v}}\|{\eta_x^\ve}\|_{L^1_{x}}\|{wf_0}\|_{H^2_xL^2_v}\le \|{wf_0}\|_{H^2_xL^2_v}.
	\end{align*}
We first consider the case $q=0$ in \eqref{w2}. 
In order to obtain the solution, we consider the following vanishing problem:
	\begin{multline}\label{17}
	(\partial_tf,g)_{L^2_{x,v}} +\ve\sum_{|\alpha|+|\beta|\le 2}(\<v\>^{4}\partial^\alpha_\beta f,\partial^\alpha_\beta g)_{L^2_{x,v}}+(v\cdot \nabla_xf,g)_{L^2_{x,v}}\\ 
	+(- A f,g)_{L^2_{x,v}} = (\Gamma(h,f),g)_{L^2_{x,v}}+(Kh,g)_{L^2_{x,v}}. 
\end{multline}
Then we denote \eqref{17} by 
\begin{align}\label{17a}
	(\partial_tf,g)_{L^2_{x,v}}+ \mathbf{B}[f,g]=(Kh,g)_{L^2_{x,v}},
\end{align}
 where $\mathbf{B}[f,g]$ is a bilinear operator on $\mathcal{H}\times\mathcal{H}$ with 
\begin{align*}
	\mathcal{H} = \{f\in L^2_{x,v}:\ &\<v\>^{2}\partial^\alpha_\beta f\in L^2_{x,v},\ \forall\,|\alpha|+|\beta|\le 2,\ {f}(t,x,v) = {f}(t,x,R_xv)\ \text{ on $\gamma_-$}\},
\end{align*} 
equipped with norm $\sum_{|\alpha|+|\beta|\le 2}\|\<v\>^2\partial^\alpha_\beta f\|_{L^2_xL^2_v}$. 
Note that for $f\in \mathcal{H}$, 
\begin{align*}
	(v\cdot\na_x f,f)_{L^2_{x,v}} &= \int_{\pa\Omega}\int_{\R^3}v\cdot n(x) |f(v)|^2\,dvdS(x)\\
	&= \int_{\pa\Omega}\int_{\R^3}R_xv\cdot n(x) |f(R_xv)|^2\,dvdS(x)\\
	&= \int_{\pa\Omega}\int_{\R^3}-v\cdot n(x) |f(v)|^2\,dvdS(x) = 0,
\end{align*}
where we used \eqref{154}$_3$ and $R_xv\cdot n(x)=-v\cdot n(x)$. 
Together with Lemma \ref{L2} and \ref{gam}, under the smallness of \eqref{156a}, we have that for $f\in \mathcal{H}$, 
\begin{align*}\notag
	\int^T_0\mathbf{B}[f,f]\,dt&\ge \ve\sum_{|\alpha|+|\beta|\le 2}\int^T_0\|\<v\>^{2}\partial^\alpha_\beta f\|^2_{L^2_{x,v}}\,dt +\int^T_0\|f\|_{L^2_xL^2_D}^2\,dt\\
	&\notag\qquad - \int^T_0\big(\|h\|_{L^\infty_xL^2_x}\|f\|_{L^2_xL^2_D}+\|h\|_{L^\infty_xL^2_D}\|f\|_{L^2_xL^2_v}\big)\|f\|_{L^2_xL^2_D}\,dt\\
	&\notag\ge {\ve}\sum_{|\alpha|+|\beta|\le 2}\int^T_0\|\<v\>^{2}\partial^\alpha_\beta f\|^2_{L^2_{x,v}}\,dt + \frac{1}{2}\int^T_0\|f\|_{L^2_xL^2_D}^2\,dt- \ve_0\sup_{0\le t\le T}\|f\|^2_{L^2_{x,v}},
\end{align*}
by choosing $\ve_0$ small enough, where we used embedding $\|\cdot\|_{L^\infty_x}\lesssim \|\cdot\|_{H^2_x}$. 
Now taking integral of \eqref{17a} over $t\in[0,T]$ with $g=f$ and letting $\ve_0$ small enough, we have 
\begin{align}\label{19}\notag
	\frac{1}{4}\sup_{0\le t\le T}\|f\|^2_{L^2_xL^2_v} + {\ve}\int^T_0\|\<v\>^{2} f\|^2_{H^2_xL^2_v}&\,dt + \frac{1}{2}\int^T_0\|f\|_{L^2_xL^2_D}^2\,dt\\
	\notag &\le \|f_0\|^2_{L^2_xL^2_v} + C\int^{T}_0\|h\|^2_{H^2_xL^2_v}\,dt\\
	&\le\|f_0\|^2_{L^2_xL^2_v}+ CT\|h\|^2_{L^\infty_{T}H^2_xL^2_v}. 
\end{align}
Note that $K$ is bounded on $L^2_v$. 
Then by the standard existence and uniqueness for linear evolution equation; cf. \cite{Evans2010}, there exists $T_0>0$ and unique solution $f^\ve\in\mathcal{H}$ to equation 
\begin{multline}\label{18}
	\int^{T_0}_0(\partial_tf^\ve,g)_{L^2_{x,v}}\,dt+ \ve\sum_{|\alpha|+|\beta|\le 2}\int^{T_0}_0(\<v\>^{4}\partial^\alpha_\beta f^\ve,\partial^\alpha_\beta g)_{L^2_{x,v}}\,dt+\int^{T_0}_0(v\cdot \nabla_xf^\ve,g)_{L^2_{x,v}}\,dt\\ 
	+\int^{T_0}_0(- A f^\ve,g)_{L^2_{x,v}}\,dt = \int^{T_0}_0(\Gamma(h,f^\ve),g)_{L^2_{x,v}}\,dt=\int^{T_0}_0(Kh,g)_{L^2_{x,v}}\,dt,
\end{multline} on $[0,T_0]\times\Omega\times\R^3$, for any test function $g\in \mathcal{H}$.
 Thanks to \eqref{19}, the sequence $\{f^\ve\}$ satisfies 
 \begin{align}\label{19aa}
 	\frac{1}{4}\sup_{0\le t\le T_0}\|f^\ve\|^2_{L^2_xL^2_v} + \frac{1}{2}\int^{T_0}_0\|f^\ve\|_{L^2_xL^2_D}^2\,dt 
 	\le\|f_0\|^2_{L^2_xL^2_v}+ CT_0\|h\|^2_{L^\infty_{T_0}H^2_xL^2_v}. 
 \end{align}
 Therefore, $\{f^\ve\}$ is uniformly bounded in $L^\infty_{T_0}L^2_xL^2_v$ and $L^2_{T_0}L^2_xL^2_D$ and hence has a weak limit $f\in L^\infty_{T_0}L^2_xL^2_v\cap L^2_{T_0}L^2_xL^2_D$. Taking weak limit $\ve\to 0$ in \eqref{18}, we have 
\begin{multline}\label{18a}
	\int^{T_0}_0(\partial_tf,g)_{L^2_{x,v}}\,dt +\int^{T_0}_0(v\cdot \nabla_xf,g)_{L^2_{x,v}}\,dt
	+\int^{T_0}_0(- A f,g)_{L^2_{x,v}}\,dt \\= \int^{T_0}_0(\Gamma(h,f),g)_{L^2_{x,v}}\,dt+\int^{T_0}_0(Kh,g)_{L^2_{x,v}}\,dt,
\end{multline}
with initial data $f(0)=f_0$, for any sufficiently smooth $g$ satisfying ${g}(t,x,v) = {g}(t,x,R_xv)$ on $\gamma_-$.
	
	Next we derive the identities on derivative. Let $|\alpha|=1$. We consider the equation 
	\begin{align*}
		\partial_tf^{\alpha} + v\cdot \nabla_xf^{\alpha}  - A \partial^{\alpha} f = \Gamma(\partial^{\alpha} h,f)+\Gamma(h,f^{\alpha} )+ K\partial^{\alpha} h,
	\end{align*}
	with initial data $f^{\alpha}(0,x,v)=\partial^{\alpha} f_0(x,v)$.
	Since $|\alpha|=1$, we assume $\alpha_i=1$ for some $i=1,2,3$. 
	Then we use Hilbert space 
	\begin{align*}
		\mathcal{H}_{\alpha} = \big\{f\in L^2_{x,v}:\ &\<v\>^{2} f\in H^2_xL^2_v,\ f\in L^2_xL^2_D, \\ &{f}(t,x,v) = -{f}(t,x,R_xv)\ \text{ on $\Gamma_i$ if $\alpha_i=1$},\\
		&{f}(t,x,v) = {f}(t,x,R_xv)\ \text{ on $\Gamma_j$ if $\alpha_j=0$ or $2$}
\big\}.
	\end{align*} Using the same argument we used to derive \eqref{18}, there exists $f^{\ve,\alpha}\in\mathcal{H}_{\alpha}$ such that 
\begin{align*}\notag
	&\quad\,\int^{T_0}_0(\partial_tf^{\ve,\alpha},g)_{L^2_{x,v}}\,dt+ \ve\sum_{|\alpha'|+|\beta|\le 2}\int^{T_0}_0(\<v\>^{4}\partial^{\alpha'}_\beta  f^{\ve,\alpha},\partial^\alpha_\beta  g)_{L^2_{x,v}}\,dt\\
	\notag&+\int^{T_0}_0(v\cdot \nabla_xf^{\ve,\alpha},g)_{L^2_{x,v}}\,dt 
	+\int^{T_0}_0(- A f^{\ve,\alpha},g)_{L^2_{x,v}}\,dt\\ &=\int^{T_0}_0\Gamma(\partial^{\alpha} h,f)\,dt+ \int^{T_0}_0(\Gamma(h,f^{\ve,\alpha}),g)_{L^2_{x,v}}\,dt+\int^{T_0}_0(K\partial^{\alpha}h,g)_{L^2_{x,v}}\,dt,
\end{align*}for $g\in\mathcal{H}_{\alpha}$.
Similar to \eqref{19}, we can obtain the energy estimate for sufficiently small $T_0$:
\begin{multline}\label{23}
	\frac{1}{4}\sup_{0\le t\le T_0}\| f^{\ve,\alpha}\|^2_{L^2_xL^2_v} + {\ve}\int^{T_0}_0\|\<v\>^{2} f^{\ve,\alpha}\|^2_{H^2_xL^2_v}\,dt + \frac{1}{2}\int^{T_0}_0\| f^{\ve,\alpha}\|_{L^2_xL^2_D}^2\,dt 
\\
	\le\|\partial^{\alpha} f_0\|^2_{L^2_xL^2_v} + T_0\|h\|^2_{L^\infty_{T_0}H^2_xL^2_v} \\+ \big(\|\partial^{\alpha} h\|_{L^\infty_{T_0}L^2_xL^2_v}\|f\|_{L^2_{T_0}L^2_xL^2_D}+\|\partial^{\alpha} h\|_{L^2_{T_0}L^2_xL^2_D}\|f\|_{L^\infty_{T_0}L^2_xL^2_v}\big)\|f^{\ve,\alpha}\|_{L^2_{T_0}L^2_xL^2_D}. 
\end{multline}
Choosing $\ve_0>0$ in \eqref{156a} small enough and using \eqref{19aa}, we have 
\begin{align}\label{23a}
\frac{1}{4}\sup_{0\le t\le T_0}\|f^{\ve,\alpha}\|^2_{L^2_xL^2_v} +  \frac{1}{2}\int^{T_0}_0\|f^{\ve,\alpha}\|_{L^2_xL^2_D}^2\,dt 
\lesssim \|f_0\|^2_{H^1_xL^2_v} + T_0\|h\|^2_{L^\infty_{T_0}H^2_xL^2_v}.
\end{align}
Then $\{f^{\ve,\alpha}\}$ is bounded in $L^\infty_{T_0}L^2_xL^2_v\cap L^2_{T_0}L^2_xL^2_D$ uniformly in $\ve>0$. 
Then the weak limit $f^{\alpha}$ of $\{f^{\ve,\alpha}\}$ is the solution to 
\begin{multline}\label{30}
	\int^{T_0}_0(\partial_tf^{\alpha},g)_{L^2_{x,v}}\,dt+\int^{T_0}_0(v\cdot \nabla_xf^{\alpha},g)_{L^2_{x,v}}\,dt\
	+\int^{T_0}_0(- A f^{\alpha},g)_{L^2_{x,v}}\,dt \\=\int^{T_0}_0\Gamma(\partial^{\alpha} h,f)\,dt+ \int^{T_0}_0(\Gamma(h,f^{\alpha}),g)_{L^2_{x,v}}\,dt+\int^{T_0}_0(K\partial^{\alpha} h,g)_{L^2_{x,v}}\,dt.
\end{multline}
 Then $f^{\alpha}=\partial^{\alpha} f$ in the weak sense by using \eqref{18a} and \eqref{30}. Also, one can use \eqref{154}$_1$ to define $\partial_nf$ on the boundary and deduce \eqref{411} as in Lemma \ref{Lem24}. 
 
 For second order derivatives, we let $|\alpha|=2$ and consider equation 
	\begin{align}\label{25a}
	\partial_tf^{\alpha} + v\cdot \nabla_xf^{\alpha}  - A \partial^{\alpha} f = \sum_{|\alpha'|\le 1}\Gamma(\partial^{\alpha-\alpha'} h,\partial^{\alpha'}f)+\Gamma(h,f^{\alpha} )+ K\partial^{\alpha} h,
\end{align}
in space $\mathcal{H}_\alpha$. 
Applying the same argument for deriving \eqref{23} and \eqref{23a}, we can obtain the solution $f^{\alpha}\in\mathcal{H}_\alpha$ to \eqref{25a} with estimate
\begin{align}\label{26a}
		\frac{1}{4}\sup_{0\le t\le T_0}\|f^{\alpha}\|^2_{L^2_xL^2_v} +  \frac{1}{2}\int^{T_0}_0\|f^{\alpha}\|_{L^2_xL^2_D}^2\,dt 
	\le\|f_0\|^2_{H^2_xL^2_v} + T_0\|h\|^2_{L^\infty_{T_0}H^2_xL^2_v}. 
\end{align}
Then $f^{\alpha}=\partial^{\alpha} f$ and $f$ satisfies \eqref{158a}. Combining estimates \eqref{19aa}, \eqref{23a} and \eqref{26a}, we obtain \eqref{158a}. 
 This completes the Lemma when $q=0$.

For the estimate with weight, we let $g=w^2f$ in \eqref{17} and use Lemma \ref{L2} and \ref{gam} to obtain 
\begin{align*}
	&\quad\,\frac{1}{2}\partial_t\|wf\|^2_{L^2_xL^2_v} + {\ve}\sum_{|\alpha|+|\beta|\le 2}\|\<v\>^{2} w\partial^\alpha_\beta f\|^2_{L^2_xL^2_v} + \frac{1}{2}\|wf\|_{L^2_xL^2_D}^2\\
	&\le C\|f\|^2_{L^2_{x,v}} + \big(\|wh\|_{H^2_xL^2_v}\|wf\|_{H^2_xL^2_D}+\|wh\|_{H^2_xL^2_D}\|wf\|_{H^2_xL^2_2}\big)\|wf\|_{H^2_xL^2_D}+C\|wh\|^2_{H^2_xL^2_v}. 
\end{align*}
Taking integration on $t\in[0,T]$ and using \eqref{156a}, we have 
\begin{align*}
	&\quad\,\frac{1}{2}\sup_{0\le t\le T}\|wf\|^2_{L^2_xL^2_v} + {\ve}\sum_{|\alpha|+|\beta|\le 2}\int^T_0\|\<v\>^{2} w\partial^\alpha_\beta f\|^2_{L^2_xL^2_v}\,dt + \frac{1}{2}\int^T_0\|wf\|_{L^2_xL^2_D}^2\,dt\\
	&\le \|wf_0\|^2_{L^2_xL^2_v}+CT\|f\|^2_{L^\infty_TL^2_{x,v}} + CT\|wh\|^2_{L^\infty_TH^2_xL^2_v}. 
\end{align*}
This is an analog estimate as \eqref{19}. The term $CT\|f\|^2_{L^\infty_TL^2_{x,v}}$ can be absorbed by the left hand side if we choose $T>0$ sufficiently small. Then one can follow the same argument from \eqref{18} to \eqref{26a} to obtain the result \eqref{158a} for weighted estimates. 
This completes Lemma \ref{Lem22}. 

	\end{proof}

\begin{proof}[Proof of Theorem \ref{localsolution}]
	We now construct the approximation solution sequence $
		\{f^n(t,x,v)\}^\infty_{n=0}
	$ using the following iterative scheme:
	\begin{equation*}\left\{
		\begin{aligned}
			&\partial_tf^{n+1} + v\cdot \nabla_xf^{n+1} - A f^{n+1} = \Gamma_{\pm}(f^n,f^{n+1})+Kf^n,
			&f^{n+1}(0,x,v) = f_0(x,v),\\ 
			&{f^{n+1}}(t,x,v) = {f^{n+1}}(t,x,R_xv)\  \text{  on }v\cdot n(x)<0,
		\end{aligned}\right.
	\end{equation*}for $n=0,1,2,\cdots$, where we set $f^0(t,x,v)=f_0(x,v)$. 
	With Lemma \ref{Lem22}, it is a standard procedure to apply the induction argument to show that there exists $\ve_0>0$ and $T_0>0$ such that if 
	\begin{align*}
		\|{wf_0}\|^2_{L^\infty_{T_0}H^2_xL^2_v}\le \ve_0,
	\end{align*}
	then the approximate solution sequence $\{f^n\}$ is well-defined satisfying 
	\begin{multline*}
		\|wf^{n+1}\|_{L^\infty_{T_0}H^2_xL^2_v}+\|wf^{n+1}\|_{L^2_{T_0}H^2_{x}L^2_{D}}\le \|wf_0\|_{H^2_xL^2_v} + T_0^{1/2}\|wf^n\|_{L^\infty_{T_0}H^2_xL^2_v}\\\le\cdots\le \sum_{n=0}^\infty T_0^{n/2}\|wf_0\|_{H^2_xL^2_v}\lesssim \ve_0,
	\end{multline*}where $T_0$ is chosen to be small enough. 
	Notice that $f^{n+1}-f^n$ solves 
	\begin{multline*}
		\partial_t(f^{n+1}-f^n) + v\cdot \nabla_x(f^{n+1}-f^n)   
		- A (f^{n+1}-f^n)\\ = \Gamma(f^n,f^{n+1}-f^n)+\Gamma(f^n-f^{n-1},f^n)+K(f^n-f^{n-1}),
	\end{multline*}
	in the weak sense, for $n=1,2,3,\cdots$. Using the method for deriving \eqref{158a}, we know that $\{f^{n+1}-f^n\}$ is Cauchy sequence with respect to norms $\|f^{n+1}-f^n\|_{L^\infty_{T_0}H^2_xL^2_v}+\|f^{n+1}-f^n\|_{L^2_{T_0}H^2_{x}L^2_{D}}$. Then the limit function $f(t,x,v)$ is indeed a unique solution to \eqref{1} and \eqref{specular} satisfying estimate \eqref{169}. 
	For the positivity, we can use the argument from \cite[Lemma 8, page 416]{Guo2002a} for Landau case and \cite[Page 833]{Gressman2011} for Boltzmann case; the details are omitted for brevity. The proof of Theorem \ref{localsolution} is completed. 	
	\end{proof}

	 \section{Appendix}\label{Append}
	
	\medskip\noindent{\bf Carleman representation.}  
	Now we have a short review of Carleman representation for Boltzmann equation. One may refer to \cite{Alexandre2000,Global2019} for details. 
	For measurable function $F(v,v_*,v',v'_*)$, if any sides of the following equation is well-defined, then
	\begin{align}
		&\int_{\R^d}\int_{\mathbb{S}^{d-1}}b(\cos\theta)|v-v_*|^\gamma F(v,v_*,v',v'_*)\,d\sigma dv_*\notag\\
		&\quad=\int_{\R^d_h}\int_{E_{0,h}}\tilde{b}(\alpha,h)\1_{|\alpha|\ge|h|}\frac{|\alpha+h|^{\gamma+1+2s}}{|h|^{d+2s}}F(v,v+\alpha-h,v-h,v+\alpha)\,d\alpha dh,\label{Carleman}
	\end{align}where $\tilde{b}(\alpha,h)$ is bounded from below and above by positive constants, and $\tilde{b}(\alpha,h)=\tilde{b}(|\alpha|,|h|)$, $E_{0,h}$ is the hyper-plane orthogonal to $h$ containing the origin. 

	\bibliographystyle{amsplain}
	\bibliography{1}

\providecommand{\bysame}{\leavevmode\hbox to3em{\hrulefill}\thinspace}
\providecommand{\MR}{\relax\ifhmode\unskip\space\fi MR }
\providecommand{\MRhref}[2]{%
  \href{http://www.ams.org/mathscinet-getitem?mr=#1}{#2}
}
\providecommand{\href}[2]{#2}
\begin{thebibliography}{10}

\bibitem{Alexandre2000}
R.~Alexandre, L.~Desvillettes, C.~Villani, and B.~Wennberg, \emph{{Entropy
  Dissipation and Long-Range Interactions}}, Arch. Ration. Mech. Anal.
  \textbf{152} (2000), no.~4, 327--355.

\bibitem{Global2019}
Radjesvarane Alexandre, Fr{\'{e}}d{\'{e}}ric H{\'{e}}rau, and Wei-Xi Li,
  \emph{{Global hypoelliptic and symbolic estimates for the linearized
  Boltzmann operator without angular cutoff}}, J. Math. Pures Appl.
  \textbf{126} (2019), 1--71.

\bibitem{Cao2019}
Yunbai Cao, Chanwoo Kim, and Donghyun Lee, \emph{{Global Strong Solutions of
  the Vlasov{\textendash}Poisson{\textendash}Boltzmann System in Bounded
  Domains}}, Arch. Ration. Mech. Anal. \textbf{233} (2019), no.~3, 1027--1130.

\bibitem{Cercignani1992}
Carlo Cercignani, \emph{{On the initial-boundary value problem for the
  Boltzmann equation}}, Arch. Ration. Mech. Anal. \textbf{116} (1992), no.~4,
  307--315.

\bibitem{Ciarlet2013}
Philippe~G. Ciarlet, \emph{{Linear and Nonlinear Functional Analysis with
  Applications: With 401 Problems and 52 Figures}}, CAMBRIDGE, 2013.

\bibitem{Duan2020}
Renjun Duan, Shuangqian Liu, Shota Sakamoto, and Robert~M. Strain, \emph{Global
  mild solutions of the landau and non-cutoff boltzmann equations}, Commun.
  Pure Appl. Math. \textbf{74} (2020), no.~5, 932--1020.

\bibitem{Duan2013a}
Renjun Duan, Shuangqian Liu, Tong Yang, and Huijiang Zhao, \emph{{Stability of
  the nonrelativistic Vlasov-Maxwell-Boltzmann system for angular non-cutoff
  potentials}}, Kinetic {\&} Related Models \textbf{6} (2013), no.~1, 159--204.

\bibitem{Esposito2013}
R.~Esposito, Y.~Guo, C.~Kim, and R.~Marra, \emph{{Non-Isothermal Boundary in
  the Boltzmann Theory and Fourier Law}}, Commun. Math. Phys. \textbf{323}
  (2013), no.~1, 177--239.

\bibitem{Evans2010}
Lawrence~C. Evans, \emph{{Partial Differential Equations: Second Edition}},
  Graduate Studies in Mathematics, American Mathematical Society, 2010.

\bibitem{Gressman2011}
Philip~T. Gressman and Robert~M. Strain, \emph{{Global classical solutions of
  the Boltzmann equation without angular cut-off}}, J. Amer. Math. Soc.
  \textbf{24} (2011), no.~3, 771--771.

\bibitem{Grisvard1985}
P~Grisvard, \emph{Elliptic problems in nonsmooth domains}, Pitman Advanced Pub.
  Program, Boston, 1985.

\bibitem{Guo2002a}
Yan Guo, \emph{{The Landau Equation in a Periodic Box}}, Commun. Math. Phys.
  \textbf{231} (2002), no.~3, 391--434.

\bibitem{Guo2009}
\bysame, \emph{{Decay and Continuity of the Boltzmann Equation in Bounded
  Domains}}, Arch. Ration. Mech. Anal. \textbf{197} (2009), no.~3, 713--809.

\bibitem{Guo2020}
Yan Guo, Hyung~Ju Hwang, Jin~Woo Jang, and Zhimeng Ouyang, \emph{{The Landau
  Equation with the Specular Reflection Boundary Condition}}, Arch. Ration.
  Mech. Anal. \textbf{236} (2020), no.~3, 1389--1454.

\bibitem{Guo2016}
Yan Guo, Chanwoo Kim, Daniela Tonon, and Ariane Trescases, \emph{{Regularity of
  the Boltzmann equation in convex domains}}, Invent. Math. \textbf{207}
  (2016), no.~1, 115--290.

\bibitem{Hamdache1992}
K.~Hamdache, \emph{{Initial-Boundary value problems for the Boltzmann equation:
  Global existence of weak solutions}}, Arch. Ration. Mech. Anal. \textbf{119}
  (1992), no.~4, 309--353.

\bibitem{Kim2011}
Chanwoo Kim, \emph{{Formation and Propagation of Discontinuity for Boltzmann
  Equation in Non-Convex Domains}}, Commun. Math. Phys. \textbf{308} (2011),
  no.~3, 641--701.

\bibitem{Kim2017}
Chanwoo Kim and Donghyun Lee, \emph{{The Boltzmann Equation with Specular
  Boundary Condition in Convex Domains}}, Commun. Pure Appl. Math. \textbf{71}
  (2017), no.~3, 411--504.

\bibitem{Liu2016}
Shuangqian Liu and Xiongfeng Yang, \emph{{The Initial Boundary Value Problem
  for the Boltzmann Equation with Soft Potential}}, Arch. Ration. Mech. Anal.
  \textbf{223} (2016), no.~1, 463--541.

\bibitem{Liu2006}
Tai-Ping Liu and Shih-Hsien Yu, \emph{{Initial-boundary value problem for
  one-dimensional wave solutions of the Boltzmann equation}}, Commun. Pure
  Appl. Math. \textbf{60} (2006), no.~3, 295--356.

\bibitem{Mischler2000}
St$\acute{e}$phane Mischler, \emph{{On the Initial Boundary Value Problem for
  the Vlasov-Poisson-Boltzmann System}}, Commun. Math. Phys. \textbf{210}
  (2000), no.~2, 447--466.

\bibitem{Pao1974}
Young-Ping Pao, \emph{{Boltzmann collision operator with inverse-power
  intermolecular potentials, I}}, Commun. Pure Appl. Math. \textbf{27} (1974),
  no.~4, 407--428.

\bibitem{Adams2003}
John J. F.~Fournier Robert~Adams, \emph{{Sobolev Spaces}}, Elsevier LTD,
  Oxford, 2003.

\bibitem{Strain2012}
Robert~M. Strain, \emph{{Optimal time decay of the non cut-off Boltzmann
  equation in the whole space}}, Kinetic {\&} Related Models \textbf{5} (2012),
  no.~3, 583--613.

\bibitem{Strain2007}
Robert~M. Strain and Yan Guo, \emph{{Exponential Decay for Soft Potentials near
  Maxwellian}}, Arch. Ration. Mech. Anal. \textbf{187} (2007), no.~2, 287--339.

\bibitem{Yang2016}
Tong Yang and Hongjun Yu, \emph{{Spectrum Analysis of Some Kinetic Equations}},
  Arch. Ration. Mech. Anal. \textbf{222} (2016), no.~2, 731--768.

\bibitem{Yang2005}
Tong Yang and Hui-Jiang Zhao, \emph{{A Half-space Problem for the Boltzmann
  Equation with Specular Reflection Boundary Condition}}, Commun. Math. Phys.
  \textbf{255} (2005), no.~3, 683--726.

\end{thebibliography}

\end{document}